\documentclass[11pt]{amsart}
\usepackage{}
\usepackage{mathrsfs}
\usepackage{amsfonts}
\usepackage{color}
\usepackage{amssymb,amsmath,amsthm,amsmath,hyperref,txfonts}

\usepackage{graphicx,float}

\numberwithin{equation}{section}

\allowdisplaybreaks

\newtheorem{thm}{Theorem}[section]

\newtheorem{lem}{Lemma}[section]
\newtheorem{rem}{Remark}[section]
\newtheorem{prop}{Proposition}[section]
\newtheorem{defn}{Definition}[section]

\newcommand{\beq}{\begin{eqnarray}}
\newcommand{\eeq}{\end{eqnarray}}
\newcommand{\beqno}{\begin{eqnarray*}}
\newcommand{\eeqno}{\end{eqnarray*}}
\newcommand{\be}{\begin{equation}}
\newcommand{\ee}{\end{equation}}
\newcommand{\beno}{\begin{equation*}}
\newcommand{\eeno}{\end{equation*}}

\newcommand\dl{\delta}
\newcommand\tl{\tilde}

\newcommand\Dl{\Delta}

\newcommand\nb{\nabla}
\newcommand\nn{\nonumber}
\newcommand{\R}{\mathbb{R}}

\newcommand\eps{\epsilon}

\newcommand\fr{\frac}
\newcommand\al{\alpha}
\newcommand{\dv}{\mathrm{div}}

\def\la{\langle}
\def\ra{\rangle}
\newcommand\les{\lesssim}
\newcommand\lm{\lambda}

\newcommand\pr{\partial}

\allowdisplaybreaks

\begin{document}
\title[compressible  Navier-Stokes equations]{Linear stability of the Couette flow in the 3D isentropic compressible  Navier-Stokes equations}
\author{Lan Zeng, Zhifei Zhang, Ruizhao Zi}

\address{ School of Mathematical Science, Peking University,
Beijing, 100871, P.R. China}
\email{zenglan1206@126.com}

\address{ School of Mathematical Science, Peking University, Beijing,
100871, P.R. China}
\email{zfzhang@math.pku.edu.cn}

\address{School of Mathematics and Statistics \& Hubei Key Laboratory of Mathematical  Sciences, Central China Normal University, Wuhan,  430079,  P. R. China.}
\email{rzz@mail.ccnu.edu.cn}


\begin{abstract}
Consider the linear stability of the three dimensional isentropic compressible Navier-Stokes equations on $\mathbb{T}\times\mathbb{R}\times\mathbb{T}$. We prove the enhanced dissipation phenomenon for the linearized isentropic compressible Navier-Stokes equations around the Couette flow $(y, 0, 0)^\top$. Moreover, the lift-up phenomenon is also shown in this paper. Compared with the 3D incompressible Navier-Stokes equations [Ann. of Math.,185(2017), 541--608], the lift-up effect here is stronger due to the loss of the incompressible condition. 
\end{abstract}
\date{\today}
\maketitle
\section{Introduction}
The motion of the compressible fluids can be described by the isentropic  compressible Navier-Stokes:
\begin{eqnarray}\label{1.1}
\left\{ \begin{array}{ll}
\pr_t\rho+\textrm{div}(\rho v)=0,\\[2mm]
\pr_t(\rho v)+\textrm{div}(\rho v\otimes v)+\fr{\nabla
P(\rho)}{\eps^2}=\mu\Dl v+(\lm+\mu)\nb\mathrm{div}v,\\
 \end{array}
 \right.
\end{eqnarray}
where the unknowns $\rho=\rho(t,x,y,z), v=v(t,x,y,z)$ with $(x, y, z)\in\mathbb{T}\times\mathbb{R}\times\mathbb{T}$ are the density and velocity, respectively. $P=P(\rho)$ is the pressure, which is a smooth function of $\rho$. The parameter  $\eps$  is the Mach number, and the constants $\mu$ and $\lm$ are the viscous coefficients satisfying the physical restriction
\be\label{visco}
\mu>0, \quad 2\mu+3\lm\ge0.
\ee



It is clear that $(\rho_s, v_s):=(1, U(y))$ with $U(y):=(y,0,0)^\top$ the Couette flow is a stationary solution to the isentropic compressible Navier-Stokes equations \eqref{1.1}. In this paper, we will investigate the stability of the stationary  solution $(\rho_s, v_s)$ to \eqref{1.1} at the linear level.
Let $(b, u)$ be the perturbation of $(\rho, v)$ around $(\rho_s, v_s)$ , i.e.,
\beno
b:=\rho-1, \quad u=v-(y,0,0)^\top.
\eeno
Assume w.l.o.g. that $P'(1)=1$, then the linearized system of \eqref{1.1} around $(\rho_s, v_s)$, the unknowns still denoted by $(b, u)$, read as follows:
\be\label{Li}
\begin{cases}
\pr_tb+y\pr_xb+\mathrm{div}u=0,\\[2mm]
\pr_tu+y\pr_xu+(u^2, 0, 0)^\top+
\fr{\nb b}{\eps^2}-\mu\Dl u-(\lm+\mu)\nb\mathrm{div}u=0.
\end{cases}
\ee
\subsection{Background and previous work}
To begin with, let us give a short review of the  extensive mathematical results on  the compressible Navier-Stokes equations in the absence of shear flow. To  the best of our knowledge, the global classical  solutions for 3D flows were first obtained by Matsumura and Nishida \cite{MN80} with initial data close to a constant state. The requirement of smoothness, but not smallness, was subsequently relaxed by Hoff  \cite{Hoff95, Ho97}. In the scaling invariant approach,  Danchin \cite{Danchin00}  constructed the global  solution with initial perturbation around constant state lying in critical space, see \cite{CD10, CMZ10, DM17, FZZ18, Ha11, Li20} for further developments. For arbitrary initial data, the breakthrough was made by Lions \cite{Lions98}, where the existence of large global  weak solutions was proved for the first time, with the adiabatic constant $\gamma\ge\fr{9}{5}$. The restriction on $\gamma$ was later relaxed by Jiang and Zhang \cite{JZ01} to $\gamma>1$ for spherically symmetric data, and by Feireisl,  Novotn\'{y} and Petzeltov\'{a} \cite{FNP01} to $\gamma>\fr32$ for general data. A new compactness approach was developed by Bresch and Jabin \cite{BJ18} to deal with  more general pressure laws and anisotropic viscous stress tensor.


The study of the linear stability of the Couette flow goes back to the classical results of Rayleigh \cite{Ray80} and Kelvin\cite{Kel87} for the incompressible fluid. Mathematically, the Couette flow is known to be spectrally stability for all Reynolds numbers \cite{DR81, Rom73}. On the contrary,  experiments  \cite{Cha02, GG94, OK80, SH01,Yag12} show instability of Couette flow and transition to turbulence for sufficient high Reynolds numbers. This  paradox gained  quite some attention in the fluid mechanics, see \cite{Cas60,  MP77, Orr07, TTR93} for instance. In particular, the transient growth was explained by the non-normality of the linearized operator in \cite{TTR93}. It is worthy pointing out that the velocity will tend to 0 as $t\rightarrow\infty$ even for a time reversible system such as the incompressible Euler equation. This phenomenon is referred to as inviscid damping, due to the relationship with Landau damping in the Vlasov equations \cite{Lan46, MV11}. The nonlinear inviscid damping was first rigorously confirmed by Bedrossian and Masmoudi \cite{BM15} on the domain $\mathbb{T}\times\mathbb{R}$. We refer to \cite{DM18, IJ20, LZ11} and references therein for more interesting results.

For the viscous fluids around the Couette flow, another stability  mechanism, the enhanced dissipation plays an important role in hydrodynamic stability. A problem of great interest related to this phenomenon is  to find the transition threshold. Significant progress has been made by Bedrossian, Germain and Masmoudi \cite{BGM20, BGM15, BGM17} in this direction for the domain $\mathbb{T}\times\mathbb{R}\times\mathbb{T}$. More precisely, in \cite{BGM17}, the authors proved that if the initial perturbation satisfies $\|u_{\mathrm{in}}\|_{H^s}\le\dl \mathrm{Re}^{-\fr{3}{2}}$, with $\mathrm{Re}$ the Reynolds number, $\dl=\dl(s)>0$, and $s>\fr92$, then the solution remains within $O(\mathrm{Re}^{-\fr12})$ of the Couette flow $(y,0,0)^\top$ in $L^2$ for all time, and converges to the streaks solution for $t\gg \mathrm{Re}^{\fr13}$. Recently, Wei and the second author \cite{WZ20} proved the nonlinear stability of the Couette flow for the initial perturbation satisfying $\|u_{\mathrm{in}}\|_{H^2}\le\dl \mathrm{Re}^{-1}$. This result implies that the transition threshold may be insensitive to the regularity of the perturbation above $H^2$. For the 2D incompressible Navier-Stokes equations, the threshold is smaller due do the absence of the lift-up effect. On the domain $\mathbb{T}\times\mathbb{R}$, it was shown in \cite{BVW18} that the threshold is not larger than $\fr12$ for the perturbation of the initial vorticity in $H^s, s>1$. Very recently, Masmoudi and Zhao \cite{MZ20} proved that the perturbation can lie in the almost critical space $H^{\log}_xL^2_y$. Moreover, they showed \cite{MZ19} that the threshold in \cite{BVW18} can be improved to not larger than $\fr13$ for more regular perturbation in Sobolev space. If the boundary is taken into consideration, the problem become much more complicated. We refer to  \cite{Bra04, KLH94, LK02, Rom73} for some rough bounds on the transition threshold. In a recent work \cite{CLWZ20},  Chen, Li, Wei and the second author  proved that in the finite channel $\mathbb{T}\times[-1,1]$  if the initial perturbation satisfies $\|u_{\mathrm{in}}\|_{H^2}\le c_0\mathrm{Re}^{-\fr12}$ for some $c_0 > 0$, then the solution  remains within $O(\mathrm{Re}^{-\fr12})$ of the Couette  flow in $L^\infty$ for all time. It is a challenging task to study the transition threshold problem of the Couette flow in the three dimensional finite channel $\mathbb{T}\times[-1,1]\times\mathbb{T}$.  Despite  the boundary layer effect, it was shown by Chen, Wei and the second author in a very rencent work \cite{CWZ20} that the threshold is still not larger than 1 (just as in the  case without boundary \cite{WZ20}), and thus the conjecture proposed by Trefethen et al \cite{TTR93} was confirmed.

Despite the long history of the stability analysis of compressible flow \cite{LL46}, the compressible stability of Couette  flow is  less well understood than the corresponding incompressible case. Glatzel \cite{Gla88, Gla89} studied the linear inviscid and viscous stability properties of the compressible plane Couette flow via a normal mode analysis in simplified flow model with constant viscous coefficients and a constant density profile. Then Duck et al. \cite{DEH94} investigated more realistic  compressible flow models where the temperature was taken into account. A transient growth mechanism was numerically studied  by Hanifi et al. in \cite{HSS96}, where they found that  the maximum transient growth scales with $\mathrm{Re^2}$ and that the time at which this maximum is reached scales with $\mathrm{Re}$. They also
pointed out this transient growth mechanism  may  physically stems from the  the {\em lift-up effect} as that in incompressible shear flows. More general  results can be found in \cite{HZ98, MDA08}, and see also \cite{GE17} for  further developments recently. An  nonmodal approach was given by Chagelishvili et al. to  consider the inviscid stability of the 2D Couette flow in \cite{Cha94, Cha97}. By means of some formal approximation, they showed that the energy of acoustic perturbations grows linear in time due to the transfer of energy from the mean flow to perturbations. When the viscosity is present,  Farrell and Ioannou \cite{FI00} showed that the inviscid growth is not sustained as $t\rightarrow\infty$ because viscous damping then takes over. Nevertheless, the transient energy growth can last longer period of time than that in  incompressible flow before being strongly affected by viscosity. As for the rigorous mathematical results, Kagei \cite{Kag11} proved that the plane Couette flow in an infinite layer is asymptotically stable if the Reynolds and Mach numbers are sufficiently small. He also showed that similar results hold for more general parallel flows \cite{Kag11-2, Kag12}. Later on, Li and Zhang \cite{LZ17} proved  that the Navier-slip boundary condition at the bottom plays a stabilizing role, and  can be used to relaxed the smallness restriction on the Reynolds number in \cite{Kag11}. Recently, Kagei and Teramoto \cite{KT20} studied the spectrum of the linearized operator around the Couette flow of the compressible Navier-Stokes equations between two concentric rotating cylinders. Finally, we would like to remark that the {\em inviscid damping} and {\em enhanced dissipation} effects were obtained by Antonelli, Dolce and Marcati \cite{ADM20,ADM21} for the two dimensional inviscid and viscous compressible Couette flow on $\mathbb{T}\times\mathbb{R}$. To our best knowledge, the stabilizing and destabilizing effects such as the enhanced dissipation and lift-up effects are still unclear for the 3D compressible Couette flow.

\subsection{The main results} The aim of this paper is to  rigorously study the stability properties of the Couette flow for the 3D linearized isentropic compressible Navier-Stokes equations on $\mathbb{T}\times\mathbb{R}\times\mathbb{T}$. We confirm the enhanced dissipation and the lift-up phenomena which have been deeply studied for the 3D incompressible fluids \cite{BGM17, WZ20}.

To state our results, let us denote
\be
f_0(y, z)=\fr{1}{2\pi}\int_{\mathbb{T}}f(x,y,z)dx, \quad f_\neq=f-f_0,
\ee
and
\be
P_{l=0}f(x,y)=\fr{1}{2\pi}\int_{\mathbb{T}}f(x,y,z)dz, \quad P_{l\neq0}f=f-P_{l=0}f.
\ee
Moreover, we simply write 
\be
f_{00}=P_{l=0}f_0, \quad\mathrm{and}\quad\tl{u}=(u^2, u^3)^\top.
\ee

Our first theorem concerns the  enhanced dissipation of  the projection of the solution to the linearized isentropic Navier-Stokes equations \eqref{Li} onto non-zero frequencies in $x$.
\begin{thm}\label{coro1}
Assume that the Mach number $\eps$ and the viscous coefficients $\mu$ and $\lm$ satisfy \eqref{visco} and
\be\label{mu}
\lm+2\mu\le1,\quad\mu^\fr13\eps\le\fr14, \quad\mathrm{and}\quad \mu^\fr13(\lm+2\mu)\eps^2\le1,
\ee
and the initial data $(b_{\mathrm{in}}, u_{\mathrm{in}})\in H^2(\mathbb{T}\times\mathbb{R}\times\mathbb{T})$. Then  there exists a constant $C_1$, depending on $\eps$, but independent of $\lm$ and $\mu$, such that
\be\label{bxz}
\|b_{\neq}(t)\|_{L^2}\le C_1\mu^{-\fr16}e^{-\fr{\mu^\fr13t}{44}}\|((\Dl b_{\mathrm{in}})_{\neq}, (\Dl u_\mathrm{in})_{\neq})\|_{L^2},
\ee

\be\label{prxxu1}
\|u^1_{\neq}(t)\|_{L^2}\le C_1e^{-\fr{\mu^\fr13t}{88}}\|((\Dl b_{\mathrm{in}})_{\neq}, (\Dl u_\mathrm{in})_{\neq})\|_{L^2},
\ee

\be\label{1.14}
\| u^{2}_\neq(t)\|_{L^2}\le C_1\mu^{-\fr13}e^{-\fr{\mu^\fr13t}{88}}\|((\Dl b_{\mathrm{in}})_{\neq}, (\Dl u_\mathrm{in})_{\neq})\|_{L^2},
\ee
and
\be\label{1.15}
\| u^3_\neq(t)\|_{L^2}\le C_1e^{-\fr{\mu^\fr13t}{44}}\|((\Dl b_{\mathrm{in}})_{\neq}, (\Dl u_\mathrm{in})_{\neq})\|_{L^2}.
\ee
\end{thm}
\begin{rem}\label{rem0}
The uniform bounds  in Theorem \ref{coro1} are obtained via higher order estimates involving good derivatives. Indeed, \eqref{bxz} also holds for $\nb _{x,z}b_\neq$, \eqref{prxxu1}  holds for $\pr_{xx}u^1$ as well, and \eqref{1.14}, \eqref{1.15} still hold with $u^2_\neq$ and $u^3_\neq$ replaced by $\Dl_{x,z}u^2_\neq$ and $\Dl_{x, z}u^3_\neq$, respectively. In particular, we would like to emphasize that the incompressibility of the quantity $\Dl u-\nb \dv u$ and  the good structure possessed by the equation of the unknown $\Dl u^2-\pr_y \dv u+\pr_xb$ (see $\eqref{Lx}_3$) play important roles in the estimate of $u^1_\neq$.
\end{rem}
The following theorem gives the enhanced dissipation of the second derivatives of  the projection of the solution to  \eqref{Li} onto non-zero frequencies in $x$. The estimates for  $\Dl u_\neq$ are given in terms of the incompressible part $\Dl u_\neq-\nb\dv u_\neq$ and the compressible part $\nb\dv u_\neq$.
\begin{thm}\label{thm-endis}
Under the condition of Theorem \ref{coro1}, there exists a constant $C_2$, depending on $\eps$, but independent of $\lm$ and $\mu$, such that
\beq\label{enhan-13}
\nn&&\mu\left(\left\|\frac{\nb \nb_{x,z}b_\neq(t)}{\eps}\right\|^2_{L^2}+\|\nb_{x,z}\dv u_\neq(t)\|^2_{L^2}+\left\|\left(\Dl u^2_\neq-\pr_y\dv u_\neq\right)(t)\right\|^2_{L^2}\right)\\
\nn&&+\mu^{\fr{4}{3}}\left\|\left(\Dl u^3_\neq-\pr_z\dv u_\neq\right)(t)\right\|^2_{L^2}\\
&\le&C_2 e^{-\fr{\mu^\fr13t}{22}}\left(\left\|\fr{(\nb{\nb_{x,z} b_{\mathrm{in}}})_{\neq}}{\eps}\right\|^2_{L^2}+\left\|(\nb \dv u_{\mathrm{in}})_{\neq}\right\|^2_{L^2}+\|(\Dl u^2_{\mathrm{in}})_\neq\|^2_{L^2}+\|(\Dl u^3_{\mathrm{in}})_\neq\|^2_{L^2}\right),
\eeq
and
\beq\label{enhan-2}
\nn&&\mu\left(\left\|\frac{\nb \pr_yb_\neq(t)}{\eps}\right\|^2_{L^2}+\|\pr_y\dv u_\neq(t)\|^2_{L^2}\right)+\mu^{\fr43}\left\|\left(\Dl u^1_\neq-\pr_x\dv u\right)(t)\right\|^2_{L^2}\\
\nn&\le&C_2 e^{-\fr{\mu^\fr13t}{44}}\Bigg\{\mu^{-\fr23}\Bigg(\left\|\fr{(\nb{\pr_{x} b_{\mathrm{in}}})_{\neq}}{\eps}\right\|^2_{L^2}+\left\|(\nb_{x,y} \dv u_{\mathrm{in}})_{\neq}\right\|^2_{L^2}+\|(\Dl u^2_{\mathrm{in}})_\neq\|^2_{L^2}\Bigg)\\
&&+\left\|\fr{(\nb{\pr_{y} b_{\mathrm{in}}})_{\neq}}{\eps}\right\|^2_{L^2}+\|(\Dl u^1_{\mathrm{in}})_\neq\|^2_{L^2}\Bigg\}.
\eeq
\end{thm}

Our last theorem is about the uniform bounds, decay estimates and the lift-up phenomenon of the projection of the solution to \eqref{Li} onto zero frequencies in $x$.
\begin{thm}\label{thm-0}
Assume that the Mach number $\eps$ and the viscous coefficients $\mu$ and $\lm$ satisfy \eqref{visco} and
\be\label{re-mu}
\lm+2\mu\le1,\quad(\lm+2\mu)\eps\le1, \quad\mathrm{and}\quad \mu(\lm+\mu)\eps^2\le1,
\ee
and the initial data $(b_{\mathrm{in}}, u_{\mathrm{in}})\in H^2(\mathbb{T}\times\mathbb{R}\times\mathbb{T})$. Then there exists a constant $C_3$, independent of $\eps, \lm$ and $\mu$, such that
\beq\label{en-0}
\nn&&\sum_{0\le|\al|\le1}\left\{\left\|\fr{\pr^{\al}b_0}{\eps}\right\|^2_{L^\infty H^1}+\|\pr^{\al}\tl{u}_0\|^2_{L^\infty H^1}+\mu\left(\left\|\fr{\nb\pr^\al b_0}{\eps}\right\|^2_{L^2L^2}+\|\nb\pr^\al\tl{u}_0\|_{L^2H^1}^2\right)\right\}\\
&\le&C_3\sum_{0\le|\al|\le1}\left\|\left(\fr{\pr^{\al}(b_{\mathrm{in}})_0}{\eps}, \pr^{\al}(\tl{u}_{\mathrm{in}})_0\right)\right\|^2_{H^1}. 
\eeq
If, in addition, the projection of the initial data onto the zero frequencies in $x$ and $z$ satisfy $((b_{\mathrm{in}})_{00}, ({u}_{\mathrm{in}})_{00})\in  L^1(\mathbb{R})$, then we have
\beq\label{decay0}
\nn&&\left\|\fr{\pr^{\al}b_0(t)}{\eps}\right\|^2_{L^2}+\|\pr^{\al}\tl{u}_0(t)\|^2_{L^2}\\
&\le&\begin{cases} 
C_3e^{-\fr13\mu t}\left\|\left(\fr{\pr^{\al}(b_\mathrm{in})_0}{\eps},\pr^{\al}(\tl{u}_\mathrm{in})_0\right)\right\|^2_{L^2}, \ \ \mathrm{if} \ \ \al_2\neq0,\\[3mm]
\fr{C_3}{\big(\mu\la t\ra\big)^{|\al_1|+\fr12}}\left(\left\|\left(\fr{\pr^{\al}(b_\mathrm{in})_0}{\eps},\pr^{\al}(\tl{u}_\mathrm{in})_0\right)\right\|^2_{L^2}+\left\|\left(\fr{\pr^{\al}(b_\mathrm{in})_{00}}{\eps},\pr^{\al}(\tl{u}_\mathrm{in})_{00}\right)\right\|^2_{L^1(\mathbb{R})}\right), \ \ \mathrm{if} \ \ \al_2=0,
\end{cases}
\eeq
and 
\beq\label{lift-up}
\nn \|u^1_0(t)\|_{L^2}+\mu^\fr12\|\nb u^1_0\|_{L^2L^2}&\leq& \|(u^1_{\mathrm{in}})_0\|_{L^2}+C_3\mu^{-1}\left(\left\|\fr{P_{l\neq0}(b_{\mathrm{in}})_0}{\eps}\right\|_{L^2}+\|P_{l\neq0}(\tl{u}_{\mathrm{in}})_0\|_{L^2}\right)\\
&&+C_3\mu^{-\fr14}\la t\ra^{\fr34}\left\|\left(\fr{({b}_{\mathrm{in}})_{00}}{\eps},({u}^2_{\mathrm{in}})_{00}\right)\right\|_{L^2\cap L^1}.
\eeq
\end{thm}
Several remarks are in order.
\begin{rem}
Our restrictions on the Mach number $\eps$, and the viscoelastic coefficient $\lm$ and $\mu$ are weaker than those in the result \cite{ADM21} for the 2D case. As a matter of fact, combining \eqref{mu} and  \eqref{re-mu} , the restrictions on $\eps, \lm$ and $\mu$ in our results can be summarized as follows
\be\label{lmmu}
\lm+2\mu\le1,\quad\eps\max\left\{\mu^\fr13, (\lm+2\mu)\right\}\le\fr14.
\ee
In \cite{ADM21}, the power on the factor $\lm+2\mu$ (the corresponding quantity should be $\lm+\mu$ in the 2D case) in the second inequality of \eqref{lmmu} is $\fr12$. Moreover, in \cite{ADM21}, a good unknown $\Xi-\nu M^2A$ (see (4.4) in \cite{ADM21}) was introduced to overcome the difficulty caused by the failure of conservation of $\Xi$. Similar circumstances will happen on ${\bf w}^1$ and ${\bf w}^2$ in our case, see the systems \eqref{Lx} and \eqref{Ly}. However, we do not resort to extra auxiliary variables, since the dissipations for ${\bf b}^2$ and ${\bf b}^1$ in \eqref{Lx} and \eqref{Ly} are available, respectively. We believe our strategy is applicable to the 2D isentropic compressible Navier-Stokes equations.
\end{rem}

\begin{rem}
Compared with the 2D result \cite{ADM21} obtained by Antonelli, Dolce and Marcati, there is no loss of derivative in the decay estimate \eqref{decay0} for $(b_0, \tl{u}_0)$. As for $u_0^1$, the energy inequality \eqref{lift-up} reveals the lift-up phenomenon. For the 3D incompressible Navier-Stokes equations, the last term in \eqref{lift-up} can be neglected due to the incompressible condition, see \cite{BGM17} for example. We would like to point out that, different from the incompressible fluids, the lift-up phenomenon also happens in the 2D compressible Navier-Stokes equations around the Couette flow. Indeed, similar estimate to \eqref{lift-up} holds for the 2D case if one ignores the $z$ variable.
\end{rem}

\begin{rem}
If the domain $\mathbb{T}\times\mathbb{R}\times\mathbb{T}$ is replaced by $\mathbb{T}\times\mathbb{R}^{2}$, our proofs are still valid. In particular, the estimates for the zero mode $(b_0, u_0)$ can be improved, since  in that case $(b_0, \tl{u}_0)$ decays as fast as the solution to the  heat equation on $\mathbb{R}^2$. More precisely, \eqref{decay0} and \eqref{lift-up} can be replaced by
\beno
\left\|\fr{\pr^{\al}b_0(t)}{\eps}\right\|^2_{L^2}+\|\pr^{\al}\tl{u}_0(t)\|^2_{L^2}\le\fr{C}{\big(\mu\la t\ra\big)^{|\al|+1}}\left\|\left(\fr{\pr^{\al}(b_\mathrm{in})_0}{\eps},\pr^{\al}(\tl{u}_\mathrm{in})_0\right)\right\|^2_{L^2\cap{L^1}},
\eeno
and 
\beqno
\nn \|u^1_0(t)\|_{L^2}+\mu^\fr12\|\nb u^1_0\|_{L^2L^2}&\leq& \|(u^1_{\mathrm{in}})_0\|_{L^2}+C\mu^{-\fr12}\log\la t\ra\left\|\left(\fr{({b}_{\mathrm{in}})_{0}}{\eps},({u}^2_{\mathrm{in}})_{0}\right)\right\|_{L^2\cap L^1},
\eeqno
respectively, for some positive constant $C$ independent of $\eps, \lm$ and $\mu$. Furthermore, we would like to remark that these improvements on the estimates for the zero mode may be helpful to solve the nonlinear stability problem. 
\end{rem}

\begin{rem}
The dependence of the Mach number $\eps$ of $C_1$ and $C_2$ mainly stems from the definitions of $m_1$ and $m_3$, where a constant $N$ depending on $\eps$  is involved. In this paper, it suffices to choose $N=O(\eps^2)$. See \eqref{up-m13} and \eqref{cN} for more details.
\end{rem}

\begin{rem}
The results in Theorems \ref{thm-0} and \ref{thm-endis} can be regard as preparations for the nonlinear problem. Note that the Sobolev embedding $H^2\hookrightarrow L^\infty$ ensures that the perturbed density $b$ and velocity $u$ are bounded, which will be useful for the nonlinear estimates.  This partially explains why the initial data $(b_{\mathrm{in}}, u_{\mathrm{in}})$ are assumed to lie in $H^2$ in this paper.  We will study the nonlinear transition threshold problem in our future work.
\end{rem}

\begin{rem}
We believe that our approach can be used to deal with more general case, for example, the non-isentropic flow that the temperature is taken into account.
\end{rem}
\noindent{\bf Notations}

\begin{enumerate}
\item Throughout the paper, we denote by $C$ various ``harmless'' positive constants  independent of the viscous coefficients $\mu, \lm,$ and the time $t$. Sometimes we use the notation $A \lesssim B$ as an equivalent to $A \le CB$. We would like to point out that $C$ may be different from line to line.

\item The Fourier transform $\hat{f}(k,\eta, l)$ of a function $f(x,y,z)$ is defined by
\beno
\hat{f}(k,\eta):=\fr{1}{(2\pi)^2}\int_\mathbb{T}\int_{\mathbb{R}}\int_\mathbb{T}f(x,y, z)e^{-i(kx+\eta y+lz)}dxdydz.
\eeno
Then 
\[
f(x,y,z)=\sum_{(k,l)\in\mathbb{Z}^2}\int_\mathbb{R}\hat{f}(k,\eta, l)e^{i\eta y}d\eta e^{i(kx+lz)}.
\]
\item The Sobolev space $H^s(\mathbb{T}\times\mathbb{R}\times\mathbb{T})$, $s\ge0$ is defined by
\beno
H^s(\mathbb{T}\times\mathbb{R}\times\mathbb{T}):=\left\{f\in L^2(\mathbb{T}\times\mathbb{R}\times\mathbb{T}):  \la D\ra^sf\in L^2(\mathbb{T}\times\mathbb{R}\times\mathbb{T})\right\},
\eeno
where
\[
\widehat{\la D\ra^sf}(k, \eta,l)=\left(1+(k^2+\eta^2+l^2)\right)^\fr{s}{2}\hat{f}(k,\eta,l).
\]

\item  For $1\le p, q\le\infty$, we simply write 
\[
L^p=L^p(\Omega), \quad  H^s=H^s(\Omega),
\]
and
\beno
\|f\|_{L^q H^s}:=\big{\|}\|f\|_{H^s(\Omega)}\big{\|}_{L^q(0,t)},
\eeno
for a function of space and time $f=f(t,\cdot)$, and the domain $\Omega$ may be taken as $\mathbb{T}\times\mathbb{R}\times\mathbb{T}, \mathbb{R}\times\mathbb{T}$, $\mathbb{R}$  or $\mathbb{R}^2$.
\item For a vector $x$, we denote $\la x\ra:=(1+|x|^2)^\fr12$. For $f, g\in L^2$, we use $\la f,g\ra$ to denote the $L^2$ inner product of $f$ and $g$.
\item For $a\in\mathbb{C}$, we use $\mathfrak{Re}\,a$ and $\bar{a}$ to denote the real part and the conjugation of $a$, respectively.
\end{enumerate}

The rest part of this paper is organized as follows. In section 2, we show the main ideas of proof. We will introduce some good unknowns and give a toy model to demonstrate  the ingredients we shall use in what follows.  In section 3, we bound the $x$-average of the solution $(b_0, u_0)$ and show the lift-up effect. The enhanced dissipation effect of the non-zere mode $(b_\neq, u_\neq)$ is shown in section 4, and  section 5 is devoted to the proof of Theorems \ref{coro1}--\ref{thm-0}.

\section{Ingredients of the proof} 
In this section, we show the ingredients of the proof. We shall give a reformulation of the linearized compressible Navier-Stokes equations \eqref{Li}. The basic idea is to decompose $\Dl u$ into the compressible part $\nb\dv u$ and the incompressible part $\Dl u-\nb \dv u$. Several good unknowns will be introduced to utilize the intrinsic cancellation structure. The proof relies on a Fourier multiplier method in the spirit of Bedrossian and Masmoudi \cite{BM15}. A toy model will be given to explain how to capture the dissipation of the perturbed density $b$ and how to design the multipliers to balance the linear stretching.
\subsection{Derivation the equations of the quantity $\Dl u-\nb\dv u$}
For the stability of 3D Couette flow in the incompressible Navier-Stokes equations, an important unknowns is $\Dl u^2$ which was first introduced by Kelvin \cite{Kel87}. Moreover,   it is convenient to work with the set of of unknowns $q^i:=\Dl u^i, i=1, 2, 3$ as observed in \cite{BGM20}. In our case, applying the $\Dl$ operator to the momentum equation in $\eqref{Li}$, and using the fact
\[
\Dl(y\pr_xu^i)=y\pr_x\Dl u^i+2\pr_{xy}u^i,
\]
we find that the equation of  $q$ takes the form of
\be\label{q}
\pr_tq+y\pr_xq+2\pr_{xy}u+(q^2, 0, 0)^\top+\fr{\nb\Dl b}{\eps^2}-\mu\Dl q-(\lm+\mu)\nb\Dl\dv u=0.
\ee
One can not see any evidence from \eqref{q} that $\Dl u^i$ behaves better than $u^i$ itself even for $i=2$.
On the other hand, we infer from $\eqref{Li}_2$  that the equation of $\nb\dv u$ satisfies
\be\label{nbdv}
\pr_t\nb\dv u+y\pr_x\nb\dv u+(0, \pr_x\dv u, 0)^\top+2\pr_x\nb u^2+\fr{\nb\Dl b}{\eps^2}-(\lm+2\mu)\nb\Dl\dv u=0.
\ee
Set
\be\label{w}
w:=q-\nb\dv u=\Dl u-\nb\dv u, 
\ee
then the different between \eqref{q} and \eqref{nbdv} shows that the equation of $w$ takes the form of 
\be\label{w}
\begin{split}
\begin{cases}
\pr_tw^1+y\pr_xw^1+w^2+\pr_y\dv u+2\pr_{xy}u^1-2\pr_{xx}u^2-\mu\Dl w^1=0,\\
\pr_tw^2+y\pr_xw^2-\pr_x\dv u-\mu\Dl w^2=0,\\
\pr_tw^3+y\pr_xw^3+2\pr_{xy}u^3-2\pr_{xz}u^2-\mu\Dl w^3=0.
\end{cases}
\end{split}
\ee
This is noting but the system of the incompressible part of the linearized compressible Navier-stokes equations \eqref{Li}. Owing to the coupling between the compressible and incompressible part, we need further reformulation of system \eqref{Li}.

\subsection{Reformulation of the equations} 
Let us denote
\beno
d:=\dv u.
\eeno
Then from $\eqref{Li}$, $\eqref{w}$, we find that  $(b, d, w^2)$ is not coupled with other components, and the system of $(b, d, w^2)$ satisfies
\begin{eqnarray}\label{L}
\left\{ \begin{array}{ll}
\pr_tb+y\pr_xb+d=0,\\[2mm]
\pr_td+y\pr_xd+2\pr_xu^2+\fr{\Dl b}{\eps^2}-(\lm+2\mu)\Dl d=0,\\[2mm]
\pr_tw^2+y\pr_xw^2-\pr_xd-\mu\Dl w^2=0,
 \end{array}
 \right.
\end{eqnarray}
with
\be\label{u2x}
\pr_xu^2=\pr_x\Dl^{-1}w^2+\pr_{xy}\Dl^{-1}d.
\ee
It's just reminiscent of the 2D linearized isentropic compressible  fluids around the Couette flow $(y, 0)$ in terms of the density, divergence of the velocity and vorticity, see \cite{ADM20, ADM21}. In order to cancel out the coupling between $w^2$ and $d$ in $\eqref{L}_3$, we derive the equation of $\pr_xb$ as follows
\be\label{bx}
\pr_t\pr_xb+y\pr_x\pr_xb+\pr_xd=0,
\ee
and define
\[
{\bf w}^2:=w^2+\pr_xb.
\]
Adding \eqref{bx} to $\eqref{L}_3$, we find that the equation of ${\bf w}^2$ takes the form of
\be
\pr_t{\bf w}^2+y\pr_x{\bf w}^2-\mu\Dl {\bf w}^2+\mu\Dl\pr_xb=0.
\ee
Denote
\[
{\bf b}^1:=\pr_xb,\quad{\rm and}\quad{\bf d}^1:=\pr_xd.
\]
Now we give the system of $({\bf b}^1, {\bf d}^1,{\bf w}^2)$
\begin{eqnarray}\label{Lx}
\left\{ \begin{array}{ll}
\pr_t{\bf b}^1+y\pr_x{\bf b}^1+{\bf d}^1=0,\\[2mm]
\pr_t{\bf d}^1+y\pr_x{\bf d}^1+2\pr_{xy}\Dl^{-1}{\bf d}^1+2\pr_{xx}\Dl^{-1}\left({\bf w}^2-{\bf b}^1\right)+\frac{\Dl {\bf b}^1}{\eps^2}-(\lm+2\mu)\Dl {\bf d}^1=0,\\[2mm]
\pr_t{\bf w}^2+y\pr_x{\bf w}^2-\mu\Dl {\bf w}^2+\mu\Dl{\bf b}^1=0.
 \end{array}
 \right.
\end{eqnarray}

Next, denote
\[
{\bf b}^3:=\pr_zb,\quad{\rm and}\quad{\bf d}^3:=\pr_zd.
\]
Note that
\beno
\pr_{xz}u^2=\pr_{xz}\Dl^{-1}\left({\bf w}^2-{\bf b}^1\right)+\pr_{xy}\Dl^{-1}{\bf d}^3,
\eeno
and
\beno
\pr_{xy}u^3=\pr_{xy}\Dl^{-1}\left(w^3+{\bf d}^3\right).
\eeno
Therefore,
\beno
\pr_{xy}u^3-\pr_{xz}u^2=\pr_{xy}\Dl^{-1}w^3-\pr_{xz}\Dl^{-1}\left({\bf w}^2-{\bf b}^1\right).
\eeno
Then it is not difficult to verify that $({\bf b}^3, {\bf d}^3, w^3)$ satisfies
\begin{eqnarray}\label{Lz}
\left\{ \begin{array}{ll}
\pr_t{\bf b}^3+y\pr_x{\bf b}^3+{\bf d}^3=0,\\[2mm]
\pr_t{\bf d}^3+y\pr_x{\bf d}^3+2\pr_{xy}\Dl^{-1}{\bf d}^3+2\pr_{xz}\Dl^{-1}\left({\bf w}^2-{\bf b}^1\right)+\frac{\Dl {\bf b}^3}{\eps^2}-(\lm+2\mu)\Dl {\bf d}^3=0,\\[2mm]
\pr_tw^3+y\pr_xw^3+2\pr_{xy}\Dl^{-1}w^3-2\pr_{xz}\Dl^{-1}\left({\bf w}^2-{\bf b}^1\right)-\mu\Dl w^3=0.
\end{array}
 \right.
\end{eqnarray}

Now we turn to reformulate the system of $(\pr_y b, \pr_yd, w^1)$. The term $\pr_yd$ in the equation $\eqref{w}_1$ of $w^1$ may cause growth and hence we  try to cancel out this bad term. To this end,  applying $\pr_y$ to $\eqref{L}_1$ yields
\be\label{by}
\pr_t\pr_yb+y\pr_x\pr_yb+\pr_xb+\pr_yd=0.
\ee
Let us define
\[
{\bf w}^1:=w^1-\pr_yb,
\]
and denote
\[
{\bf b}^2:=\pr_yb,\quad{\rm and}\quad{\bf d}^2=\pr_yd,
\]
Taking the difference between $\eqref{w}_1$ and \eqref{by}, we are led to  
\be\label{w1'}
\pr_t{\bf w}^1+y\pr_x{\bf w}^1+w^2-{\bf b}^1+2\pr_{xy}u^1-2\pr_{xx}u^2-\mu\Dl {\bf w}^1-\mu\Dl {\bf b}^2=0.
\ee
Noting that
\[
\pr_{xy}u^1=\pr_{xy}\Dl^{-1}\left({\bf w}^1+{\bf b}^2+{\bf d}^1\right),
\]
and using \eqref{u2x}, we obtain
\be
\pr_{xy}u^1-\pr_{xx}u^2=\pr_{xy}\Dl^{-1}\left({\bf w}^1+{\bf b}^2\right)-\pr_{xx}\Dl^{-1}\left({\bf w}^2-{\bf b}^1\right).
\ee
Then we find that $({\bf b}^2, {\bf d}^2, {\bf w}^1)$ solves
\begin{eqnarray}\label{Ly}
\begin{cases}
\pr_t{\bf b}^2+y\pr_x{\bf b}^2+{{\bf  b}^1}+{\bf d}^2=0,\\[2mm]
\pr_t{\bf d}^2+y\pr_x{\bf d}^2+{{\bf d}^1}+2\pr_{xy}\Dl^{-1}{\bf d}^2+2\pr_{xy}\Dl^{-1}\left({\bf w}^2-{\bf b}^1\right)+\frac{\Dl {\bf b}^2}{\eps^2}-(\lm+2\mu)\Dl {\bf d}^2=0,\\
\pr_t{\bf w}^1+y\pr_x{\bf w}^1-\mu\Dl {\bf w}^1+{{\bf w}^2-2{\bf b}^1}+2\pr_{xy}\Dl^{-1}\left({\bf w}^1+{{\bf b}^2}\right)\\
\quad\quad\quad\quad\quad\quad\quad\quad\quad\quad\quad\quad\quad\ \ \ -2\pr_{xx}\Dl^{-1}\left({\bf w}^2-{\bf b}^1\right)-\mu\Dl {\bf b}^2=0.
\end{cases}
\end{eqnarray}


\subsection{Change of coordinates} 
As mentioned above, we will use the Fourier multiplier method in this paper.  Thereby it is convenient to switch to new variables defined below 
\be\label{ch-co}
X=x-ty,\quad Y=y, \quad Z=z.
\ee
For $i\in\{1,2,3\}$, and $j\in\{1,2\}$, we define functions
\[
\begin{split}
B^i(t, X, Y, Z)=&{\bf b}^i(t, X+tY, Y, Z),\\
D^i(t, X, Y, Z)=&{\bf d}^i(t, X+tY, Y, Z),\\
W^j(t, X, Y, Z)=&{\bf w}^j(t, X+tY, Y, Z),\\
\end{split}
\]
and
\[
W^3(t, X, Y, Z)={ w}^3(t, X+tY, Y, Z).
\]
Then the systems \eqref{Lx}, \eqref{Lz} and \eqref{Ly} can be rewritten as
\begin{eqnarray}\label{Lx'}
\left\{ \begin{array}{ll}
\pr_tB^1+D^1=0,\\[2mm]
\pr_tD^1+2\pr_{XY}^L\Dl^{-1}_LD^1+2\pr_{XX}\Dl^{-1}_L\left(W^2-B^1\right)+\frac{\Dl_L B^1}{\eps^2}-(\lm+2\mu)\Dl_L D^1=0,\\[2mm]
\pr_tW^2-\mu\Dl_L W^2+\mu\Dl_LB^1=0,
 \end{array}
 \right.
\end{eqnarray}

\begin{eqnarray}\label{Lz'}
\left\{ \begin{array}{ll}
\pr_tB^3+D^3=0,\\[2mm]
\pr_tD^3+2\pr_{XY}^L\Dl^{-1}_LD^3+2\pr_{XZ}\Dl^{-1}_L\left(W^2-B^1\right)+\frac{\Dl_L B^3}{\eps^2}-(\lm+2\mu)\Dl_L D^3=0,\\[2mm]
\pr_tW^3+2\pr_{XY}^L\Dl^{-1}W^3-2\pr_{XZ}\Dl^{-1}_L\left(W^2-B^1\right)-\mu\Dl_L W^3=0,
\end{array}
 \right.
\end{eqnarray}
and
\begin{eqnarray}\label{Ly'}
\begin{cases}
\pr_tB^2+D^2=F,\\[2mm]
\pr_tD^2+2\pr_{XY}^L\Dl^{-1}_LD^2+\frac{\Dl_L B^2}{\eps^2}-(\lm+2\mu)\Dl_L D^2=G,\\[2mm]
\pr_tW^1+2\pr_{XY}^L\Dl^{-1}_L\left(W^1+B^2\right)-\mu\Dl_L W^1-\mu\Dl_LB^2=H,
\end{cases}
\end{eqnarray}
with
\be\label{FGH}
\begin{split}
F:=&-B^1,\\
G:=&-D^1-2\pr_{XY}^L\Dl_L^{-1}\left(W^2-B^1\right),\\
H:=&2B^1-W^2+2\pr_{XX}\Dl^{-1}_L\left(W^2-B^1\right).
\end{split}
\ee
Denote
\be
p=k^2+(\eta-kt)^2+l^2, \quad p'=-2k(\eta-kt).
\ee
Then 
\[
2\widehat{\pr_{XY}^L}=-2k(\eta-kt)=p',\quad 2\widehat{\pr_{XY}^L\Dl_L^{-1}}=\fr{2k(\eta-kt)}{p}=-\fr{p'}{p}.
\]
\subsection{The toy model and Fourier multipliers}
The systems \eqref{Lx'}--\eqref{Ly'} lead us to consider the following toy model (taking $\eps=1$ for simplicity):
\be\label{toy}
\begin{cases}
\pr_t\phi+\psi=0,\\
\pr_t\psi+2\pr_{XY}^L\Dl_L^{-1}\psi+\Dl_L\phi-(\lm+2\mu)\Dl_L\psi=0.
\end{cases}
\ee
In Fourier variables \eqref{toy} can be rewritten as
\be\label{toyF}
\begin{cases}
\pr_t\hat{\phi}+\hat{\psi}=0,\\
\pr_t\hat{\psi}-\fr{p'}{p}\hat{\psi}-p\hat{\phi}+(\lm+2\mu)p\hat{\psi}=0.
\end{cases}
\ee
The evolution of $|\hat{\psi}|^2$ reads
\be\label{psi2}
\pr_t\left(\fr{|\hat{\psi}|^2}{2}\right)-\fr{p'}{p}|\hat{\psi}|^2-p\mathfrak{Re}(\hat{\phi}\bar{\hat{\psi}})+(\lm+2\mu)p|\hat{\psi}|^2=0.
\ee
In order to eliminate  the coupling term $-p\mathfrak{Re}(\hat{\phi}\bar{\hat{\psi}})$, which equals to $-\mathfrak{Re}(\sqrt{p}\hat{\phi}\sqrt{p}\bar{\hat{\psi}})$, let us investigate the evolution of $\sqrt{p}\hat{\phi}$ 
\beno
\pr_t\left(\sqrt{p}\hat{\phi}\right)-\fr12\fr{p'}{p}\left(\sqrt{p}\hat{\phi}\right)+\sqrt{p}\hat{\psi}=0.
\eeno
Consequently,
\be\label{phi2}
\pr_t\left(\fr{|\sqrt{p}\hat{\phi}|^2}{2}\right)-\fr12\fr{p'}{p}|\sqrt{p}\hat{\phi}|^2+p\mathfrak{Re}(\hat{\phi}\bar{\hat{\psi}})=0.
\ee
Combining \eqref{psi2} with \eqref{phi2} yields
\be\label{par-dis}
\pr_t\left(\fr{|\hat{\psi}|^2+|\sqrt{p}\hat{\phi}|^2}{2}\right)-\fr{p'}{p}\left(|\hat{\psi}|^2+\fr12|\sqrt{p}\hat{\phi}|^2\right)+(\lm+2\mu)p|\hat{\psi}|^2=0.
\ee
We can see from this equality that $(\sqrt{p}\hat{\phi},\hat{\psi})$ is partially dissipative. Fortunately, the damping effect of $\sqrt{p}\hat{\phi}$   hidden in the system \eqref{toy} can be shown by the classical method which goes back to  Matsmura and Nishida \cite{MN80}. Indeed,  we infer from \eqref{toyF} that the product $\hat{\psi}\bar{\hat{\phi}}$ of $\hat{\psi}$ and $\bar{\hat{\phi}}$ satisfies 
\be\label{dam-phi}
\pr_{t}(\hat{\psi}\bar{\hat{\phi}})-|\sqrt{p}\hat{\phi}|^2+|\hat{\psi}|^2=\fr{p'}{p}\hat{\psi}\bar{\hat{\phi}}-(\lm+2\mu)p\hat{\psi}\bar{\hat{\phi}}.
\ee
A linear combination of \eqref{par-dis} and \eqref{dam-phi} gives
\be\label{full-dis}
\begin{split}
&\pr_t\left(\fr{|\hat{\psi}|^2+|\sqrt{p}\hat{\phi}|^2}{2}-c\hat{\psi}\bar{\hat{\phi}}\right)-\fr{p'}{p}\left(|\hat{\psi}|^2+\fr12|\sqrt{p}\hat{\phi}|^2\right)\\
&\quad\quad\quad+c|\sqrt{p}\hat{\phi}|^2+\left((\lm+2\mu)p-c\right)|\hat{\psi}|^2=\mathrm{error\ \ terms},
\end{split}
\ee
where $c$ is to be determined. Clearly, one can choose $c=\dl\mu$ with $\dl\ll1$ to ensure 
\beno
\left((\lm+2\mu)p-c\right)|\hat{\psi}|^2\gtrsim \mu p|\hat{\psi}|^2,
\eeno
and
\beno
\fr{|\hat{\psi}|^2+|\sqrt{p}\hat{\phi}|^2}{2}-c\hat{\psi}\bar{\hat{\phi}}\approx|\hat{\psi}|^2+|\sqrt{p}\hat{\phi}|^2,
\eeno
for $k\neq0$. However, if $c$ were chosen in this way, it means that we only used the dissipation of $\psi$, but ignored the enhanced dissipation of $\psi$ which is easily to be seen if one just focuses on $\pr_t\psi-(\lm+2\mu)\Dl_L\psi$ in \eqref{toy}, see \cite{BGM17}, for instance. In fact, if the enhanced dissipation of $\psi$ is taken into account,  we can regard \eqref{full-dis} as follows at least fomally
\be\label{full-dis-re}
\begin{split}
&\pr_t\left(\fr{|\hat{\psi}|^2+|\sqrt{p}\hat{\phi}|^2}{2}-c\hat{\psi}\bar{\hat{\phi}}\right)-\fr{p'}{p}\left(|\hat{\psi}|^2+\fr12|\sqrt{p}\hat{\phi}|^2\right)\\
&\quad\quad\quad+c|\sqrt{p}\hat{\phi}|^2+\left(\al(\lm+2\mu)p+\beta\mu^\fr13-c\right)|\hat{\psi}|^2=\mathrm{error\ \ terms},
\end{split}
\ee
with some positive constants $\al$ and $\beta$ satisfying $\al+\beta\le1$. At this stage, it is natural to choose
\beno
c=\dl\mu^\fr13,\quad\mathrm{with}\quad\dl\ll1.
\eeno
As a result, both the enhanced dissipation of $\psi$ and $\sqrt{-\Dl_L}\phi$ are expected to be generated.

In all the analyses above, the bad term $-\fr{p'}{p}\left(|\hat{\psi}|^2+\fr12|\sqrt{p}\hat{\phi}|^2\right)$
is overlooked. Now we turn to deal with this term. First of all, based on the choice of $c$, it is necessary to study the following simplified toy model
\be\label{sim-toy}
\pr_t\hat{f}-\fr{p'}{p}\hat{f}+\mu^\fr13\hat{f}=0.
\ee
This equation can be seen as a competition between the linear stretching term $-\fr{p'}{p}\hat{f}$ and the damping  term $\mu^\fr13\hat{f}$. Motivated by \cite{BGM17}, we introduce a multiplier $m$ to balance the linear stretching if the damping is not dominated. Roughly speaking, if the stretching overcomes the damping, we define the multiplier $m$ in such a way that $m^{-1}f$ solves
\be
\pr_t(m^{-1}\hat{f})+\mu^\fr13(m^{-1}\hat{f})=0.
\ee
To this end, let us first  compare the sizes of the  two factors $\fr{|p'|}{p}$ and $\mu^\fr13$.
Indeed, if $k\neq0$, then
\[
\fr{|k(\eta-kt)|}{k^2+(\eta-kt)^2+l^2}=\fr{\left|t-\fr{\eta}{k}\right|}{1+\left(t-\fr{\eta}{k}\right)^2+\fr{l^2}{k^2}}<\fr{1}{\left|t-\fr{\eta}{k}\right|}\ll \mu^\fr13,
\]
uniformly in $(k, \eta,l)$, as long as
\[
\left|t-\fr{\eta}{k}\right|\gg\mu^{-\fr13}.
\]
This coincides with the comparison  between $\fr{|p'|}{p}$ and $\mu p$ for the 3D incompressible case in \cite{BGM17}. Accordingly, the corresponding  multiplier constructed by Bedrossian, Germain and Masmoudi in \cite{BGM17} still applies in this paper. We will rewrite it in a form for our convenience, see also \cite{ADM21}.
\begin{defn}\label{def-m}
Let us define $m=m(t,k,\eta,l)$ by the exact formulas below:\par
\noindent{\em(i)}\ \ \ if $k=0$: $m(t, 0,\eta,l)=1$;\par
\noindent{\em(ii)} \ if $k\neq0$, $\fr{\eta}{k}\le-64\mu^{-\fr13}$: $m(t,k,\eta,l)=1$;\par
\noindent{\em(iii)} if $k\neq0, -64\mu^{-\fr13}<\fr{\eta}{k}\le0$:\par
{\em (iii.1)}\ \ if $0\le t\le \fr{\eta}{k}+64\mu^{-\fr13}, m(t, k,\eta,l):=\fr{k^2+(\eta-kt)^2+l^2}{k^2+\eta^2+l^2}$,\par
{\em (iii.2)}\ \ if $t>\fr{\eta}{k}+64\mu^{-\fr13}$, $m(t,k,\eta,l):=\fr{k^2+(64\mu^{-\fr13}k)^2+l^2}{k^2+\eta^2+l^2}$;\par
\noindent{\em(iv)} if $k\neq0, \fr{\eta}{k}>0$:\par
{\em(iv.1)}\ \ if $0\le t<\fr{\eta}{k}$, $m(t,k,\eta,l)=1$,\par
{\em(iv.2)}\ \ if $\fr{\eta}{k}\le t\le\fr{\eta}{k}+64\mu^{-\fr13}$, $m(t, k, \eta,l)=\fr{k^2+(\eta-kt)^2+l^2}{k^2+l^2}$,\par
{\em(iv.3)}\ \ if $t>\fr{\eta}{k}+64\mu^{-\fr13}$, $m(t, k, \eta,l)=\fr{k^2+(64\mu^{-\fr13}k)^2+l^2}{k^2+l^2}$.\par
\end{defn}
\begin{rem}\label{rem-m}
For $k\neq0$, the value of the multiplier $m$ can be classified into the following three situations:
\begin{itemize}
\item In case {\em(ii)}, sub-case {\em(iii.2)}, and sub-case {\em(iv.3)}, we have $t-\fr{\eta}{k}\ge64\mu^{-\fr13}$, and thus
\be\label{vis-dom}
\fr{|k(\eta-kt)|}{k^2+(\eta-kt)^2+l^2}< \fr{\mu^\fr13}{64}.
\ee
\item In sub-cases {\em(iii.1)} and {\em(iv.2)}, there holds $k(\eta-kt)\le0$, which means that the linear  term $-\fr{p'}{p}\hat{f}$ in \eqref{sim-toy} corresponds to an amplification of $\hat{f}$. Indeed, in these two sub-cases, $m$ satisfies the ODE
\be\label{def-w}
\fr{\pr_t{m}}{m}=\fr{p'}{p}.
\ee
\item In the sub-case {\em(iv.1)}, it holds that 
\be\label{p'p}
\fr{p'}{p}\le0,
\ee 
and thus  the linear  term $-\fr{p'}{p}\hat{f}$ in \eqref{sim-toy} amounts to a damping effect.
\end{itemize} 
Furthermore,  noting that $m$ is nondecreasing in $t$, we have
\[
1\le m(t,k,\eta,l)\le\fr{k^2+(64\mu^{-\fr13}k)^2+{l^2}}{k^2+{l^2}}.
\]
Accordingly, for $0<\mu\le1$, one easily deduces that
\be\label{m-1}
1\le m(t,k,\eta,l)\les\mu^{-\fr23}.
\ee
Finally, using again the fact that $m(t, k,\eta,l)$ is nondecreasing in $t$, we find that
\be\label{m-2}
m(t,k,\eta,l)\le\fr{k^2+(\eta-kt)^2+l^2}{k^2+l^2}.
\ee
\end{rem}
Several additional multipliers are also needed in this paper. In fact, we will use multipliers $m_1$ and $m_3$ to balance the growth due to the linear coupling between $\dv u$ and $w$.  Similar multipliers can be found in \cite{ADM20, BGM17, Z17}. In \cite{BGM17}, Bedrossian, Germain and Masmoudi constructed a multiplier $M^2$ to compensate for the transient slow-down of the enhanced dissipation near the critical times. The multiplier  $m_2$ below, already used in \cite{ADM21, BVW18}, is  an adaptation of $M^2$ in \cite{BGM17}.
\begin{defn}\label{def-adm}
Assume that $N$ is a positive constant. Let us define $m_i, i=1,2,3$ as follows: if $ k=0$, $m_i(t,0,\eta)=1$ and 
\begin{itemize}
\item if $ k\neq0$,
\be\label{m1}
\pr_tm_1=N\fr{k^2}{p}m_1, \quad m_1|_{t=0}=1,
\ee
\item if $ k\neq0$,
\be\label{m2}
\pr_tm_2=\fr{\mu^\fr13}{1+(\mu^\fr13|t-\eta/k|)^2}m_2, \quad m_2|_{t=0}=1,
\ee
\item if $ k\neq0$,
\be\label{m3}
\pr_tm_3=N\fr{\la kl\ra}{p}m_3, \quad m_3|_{t=0}=1.
\ee
\end{itemize}
\end{defn}
\begin{rem}
The multipliers $m_i, i=1, 2, 3$ are bounded from above and below uniformly in $\mu$ and $(t, k,\eta,l)\in[0,\infty)\times\mathbb{Z}\times\mathbb{R}\times\mathbb{Z}$. In particular, the upper bounds of $m_1$ and $m_3$ depend on the constant $N$. More precisely, direct calculations yield
\be\label{up-m13}
1\le m_i\le e^{N\pi}, \ \  i=1, 3, \quad \mathrm{and} \quad 1\le m_2\le e^{\pi}.
\ee
Moreover, for $(t,k,\eta,l)\in[0,\infty)\times\mathbb{Z}\backslash \{0\}\times\mathbb{R}\times\mathbb{Z}$, the multiplier $m_2$ satisfies
\be\label{claim}
\fr{\pr_tm_2}{m_2}+\mu p\ge\fr12\mu^\fr13.
\ee
In fact, if $|t-\eta/k|\ge\mu^{-\fr13}$, and $k\neq0$
\beno
\mu p=\mu^\fr13 \mu^{\fr23}(k^2+(\eta-kt)^2+l^2)=\mu^\fr13k^2\left(\mu^{\fr23}+\mu^\fr23|t-\eta/k|^2+\fr{l^2}{k^2}\right)>\mu^\fr13.
\eeno
On the other hand, if $|t-\eta/k|\le\mu^{-\fr13}$, and $k\neq0$
\beno
\fr{\pr_tm_2}{m_2}=\fr{\mu^\fr13}{1+(\mu^\fr13|t-\eta/k|)^2}\ge\fr{\mu^\fr13}{2},
\eeno
so the inequality \eqref{claim} is true.
\end{rem}
\section{Estimates of zero mode}
 We begin this section by giving the $x$-average of   the system \eqref{Li}:
\be\label{Li-0}
\begin{cases}
\pr_tb_0+\mathrm{div}\tl{u}_0=0,\\[2mm]
\pr_t\tl{u}_0+
\fr{\nb b_0}{\eps^2}-\mu\Dl \tl{u}_0-(\lm+\mu)\nb\mathrm{div}\tl{u}_0=0,\\[2mm]
\pr_tu_0^1+u^2_0-\mu\Dl u^1_0=0.
\end{cases}
\ee
Moreover, it will also be convenient to consider the projection of the sub-system of $(b_0, \tl{u}_0)$ in \eqref{Li-0} onto the zero frequencies in $z$:
\be\label{pro-z}
\begin{cases}
\pr_t{b}_{00}+\pr_y{u}^2_{00}=0,\\[2mm]
\pr_t{u}^2_{00}+
\fr{\pr_y {b}_{00}}{\eps^2}-(\lm+2\mu)\pr_{yy}{u}^2_{00}=0,\\[2mm]
\pr_t{u}_{00}^3-\mu\pr_{yy} {u}^3_{00}=0,
\end{cases}
\ee
where we have used $(b_{00}, u^2_{00}, u^3_{00})$ to denote $(P_{l=0}b_0, P_{l=0}u^2_0, P_{l=0}u^3_0)$.

The aim of this section is to study the long-time dynamics of $(b_0, u_0)$. One can see from \eqref{Li-0} that $(b_0, \tl{u}_0)$ satisfies a parabolic-hyperbolic system, while $u^1_0$ satisfies a inhomogeneous heat equation. Thus, $(b_0, u_0)$ and $u^1_0$ should be estimated in different way.
\subsection{Estimates of $(b_0, \tl{u}_0)$}
To begin with, we establish the energy estimates of $(b_0, \tl{u}_0)$.
\begin{prop}\label{prop-en0}
Assume that the Mach number $\eps$ and the viscous coefficients $\lm$ and $\mu$ satisfy 
\be\label{mu0}
\max\left\{(\lm+\mu)\mu\eps^2, (\lm+\mu)\eps\right\}\le1.
\ee
Then there exists a positive constant $C$, independent of $\eps, \lm$ and $\mu$, such that for any multi-index $\al=(\al_1, \al_2)$, there holds
\beq\label{en0}
\nn&&\left\|\fr{\pr^{\al}b_0}{\eps}\right\|^2_{H^1}+\|\pr^{\al}\tl{u}_0\|^2_{H^1}+\mu\left(\|\nb\pr^\al\tl{u}_0\|_{L^2H^1}^2+\left\|\fr{\nb\pr^\al b_0}{\eps}\right\|^2_{L^2L^2}\right)\\
&\le&C\left(\left\|\fr{\pr^{\al}(b_{\mathrm{in}})_0}{\eps}\right\|^2_{H^1}+\|\pr^{\al}(\tl{u}_{\mathrm{in}})_0\|^2_{H^1}\right),
\eeq
where $\pr^\al=\pr_y^{\al_1}\pr_z^{\al_2}$.  
\end{prop}
\begin{proof}
Standard energy estimates show that
\be\label{en0-en0}
\fr12\fr{d}{dt}\left(\left\|\fr{\pr^{\al}b_0}{\eps}\right\|^2_{H^1}+\|\pr^{\al}\tl{u}_0\|^2_{H^1}\right)+\mu\|\nb\pr^\al\tl{u}_0\|_{H^1}^2+(\lm+\mu)\|\dv \pr^\al\tl{u}_0\|^2_{H^1}=0,
\ee
and
\be\label{cross-en0}
\fr{d}{dt}\la\pr^\al\tl{u}_0, \nb\pr^\al b_0 \ra+\left\|\fr{\nb\pr^\al b_0}{\eps}\right\|^2_{L^2}-\|\dv\pr^\al \tl{u}_0\|^2_{L^2}=\mu\la \Dl\pr^\al\tl{u}_0, \nb\pr^\al b_0\ra+(\lm+\mu)\la \nb\dv \pr^\al\tl{u}_0, \nb\pr^\al b_0\ra.
\ee
Consequently,
\beq
\nn&&\fr12\fr{d}{dt}\left(\left\|\fr{\pr^{\al}b_0}{\eps}\right\|^2_{H^1}+\|\pr^{\al}\tl{u}_0\|^2_{H^1}+({\lm+\mu})\la\pr^\al\tl{u}_0, \nb\pr^\al b_0 \ra\right)\\
\nn&&+\mu\|\nb\pr^\al\tl{u}_0\|_{H^1}^2+\fr{\lm+\mu}{2}\|\dv \pr^\al\tl{u}_0\|^2_{H^1}+\fr{\lm+\mu}{2}\left\|\fr{\nb\pr^\al b_0}{\eps}\right\|^2_{L^2}\\
&=&\fr{\mu(\lm+\mu)}{2}\la \Dl\pr^\al\tl{u}_0, \nb\pr^\al b_0\ra+\fr{(\lm+\mu)^2}{2}\la \nb\dv \pr^\al\tl{u}_0, \nb\pr^\al b_0\ra.
\eeq
Clearly, 
\beno
({\lm+\mu})\la\pr^\al\tl{u}_0, \nb\pr^\al b_0 \ra\le\fr12\left(\|\pr^\al\tl{u}_0\|^2_{L^2}+\left\|\fr{\nb\pr^\al b_0}{\eps}\right\|^2_{L^2}\right),
\eeno
provided $\eps(\lm+\mu)\le1$,  and hence
\be\label{app0}
\left\|\fr{\pr^{\al}b_0}{\eps}\right\|^2_{H^1}+\|\pr^{\al}\tl{u}_0\|^2_{H^1}+({\lm+\mu})\la\pr^\al\tl{u}_0, \nb\pr^\al b_0 \ra\approx\left\|\fr{\pr^{\al}b_0}{\eps}\right\|^2_{H^1}+\|\pr^{\al}\tl{u}_0\|^2_{H^1}.
\ee
On the other hand, if $(\lm+\mu)\mu\eps^2\le1$, there holds
\beqno
\nn\fr{\mu(\lm+\mu)}{2}\la \Dl\pr^\al\tl{u}_0, \nb\pr^\al b_0\ra&\le&\fr{\lm+\mu}{8}\left\|\fr{\nb\pr^\al b_0}{\eps}\right\|^2_{L^2}+\fr{(\lm+\mu)\mu^2\eps^2}{2}\|\Dl\pr^\al\tl{u}_0\|^2_{L^2}\\
&\le&\fr{\lm+\mu}{8}\left\|\fr{\nb\pr^\al b_0}{\eps}\right\|^2_{L^2}+\fr{\mu}{2}\|\Dl\pr^\al\tl{u}_0\|^2_{L^2},
\eeqno
and if $(\lm+\mu)\eps\le1$, we have
\beqno
\nn\fr{(\lm+\mu)^2}{2}\la \nb\dv \pr^\al\tl{u}_0, \nb\pr^\al b_0\ra&\le&\fr{\lm+\mu}{8}\left\|\fr{\nb\pr^\al b_0}{\eps}\right\|^2_{L^2}+\fr{(\lm+\mu)^3\eps^2}{2}\|\nb\dv \pr^\al\tl{u}_0\|^2_{L^2}\\
&\le&\fr{\lm+\mu}{8}\left\|\fr{\nb\pr^\al b_0}{\eps}\right\|^2_{L^2}+\fr{\lm+\mu}{2}\|\nb\dv \pr^\al\tl{u}_0\|^2_{L^2}.
\eeqno
It follows that
\be\label{dinq0}
\fr{d}{dt}\left(\left\|\fr{\pr^{\al}b_0}{\eps}\right\|^2_{H^1}+\|\pr^{\al}\tl{u}_0\|^2_{H^1}+({\lm+\mu})\la\pr^\al\tl{u}_0, \nb\pr^\al b_0 \ra\right)+\mu\|\nb\pr^\al\tl{u}_0\|_{H^1}^2+\fr{\lm+\mu}{2}\left\|\fr{\nb\pr^\al b_0}{\eps}\right\|^2_{L^2}\le0.
\ee
Combining this with \eqref{app0}, and using \eqref{visco}, we  get \eqref{en0}. The proof of Proposition \ref{prop-en0} is completed.
\end{proof}
It is worth pointing out that the differential inequality \eqref{dinq0} implies the exponential decay of the projection of $(b_0,\tl{u}_0)$ onto the nonzero frequencies in $z$. We shall give a more accurate estimate below.
\begin{lem}
Assume that the Mach number $\eps$ and the viscous coefficients $\lm$ and $\mu$ satisfy
\be\label{mu-1}
(\lm+2\mu)\eps\le1.
\ee
Then there exists a positive constant $C$, independent of $\eps, \lm$ and $\mu$, such that for any multi-index $\al=(\al_1, \al_2)$, there holds
\be\label{decay1}
\left\|\fr{\pr^{\al}P_{l\neq0}b_0}{\eps}\right\|^2_{L^2}+\|\pr^{\al}P_{l\neq0}\tl{u}_0\|^2_{L^2}
\le Ce^{-\fr13\mu t}\left(\left\|\fr{\pr^{\al}P_{l\neq0}(b_{\mathrm{in}})_0}{\eps}\right\|^2_{L^2}+\|\pr^{\al}P_{l\neq0}(\tl{u}_{\mathrm{in}})_0\|^2_{L^2}\right).
\ee
\end{lem}
\begin{proof}
Clearly, the equality \eqref{en0-en0} still holds for $(P_{l\neq0}b_0, P_{l\neq0}\tl{u}_0)$ with the $H^1$ norm replaced by $L^2$ norm
\be\label{en0-basic}
\fr12\fr{d}{dt}\left(\left\|\fr{\pr^{\al}P_{l\neq0}b_0}{\eps}\right\|^2_{L^2}+\|\pr^{\al}P_{l\neq0}\tl{u}_0\|^2_{L^2}\right)+\mu\|\nb\pr^\al P_{l\neq0}\tl{u}_0\|_{L^2}^2+(\lm+\mu)\|\dv \pr^\al\tl P_{l\neq0}{u}_0\|^2_{L^2}=0.
\ee
Apply the operator $\Dl^{-1}\dv P_{l\neq0}$ to the equation $\eqref{Li-0}_2$ yields
\beno
\pr_t\Dl^{-1}\dv P_{l\neq0}\tl{u}_0+\fr{P_{l\neq0}b_0}{\eps^2}-(\lm+2\mu)\dv P_{l\neq0}\tl{u}_0=0.
\eeno
Then it is easy to verify that
\beq\label{cross-new}
\nn&&\fr{d}{dt}\la \pr^\al P_{l\neq0} b_0, \Dl^{-1}\dv \pr^\al P_{l\neq0}\tl{u}_0 \ra+\left\|\fr{\pr^\al P_{l\neq0} b_0}{\eps}\right\|^2_{L^2}-\left\|(-\Dl)^{-\fr12}\dv \pr^\al P_{l\neq0}\tl{u}\right\|^2_{L^2}\\&=&(\lm+2\mu)\la \dv \pr^\al P_{l\neq0}\tl{u}_0, \pr^\al P_{l\neq0} b_0\ra.
\eeq
Combining \eqref{en0-basic} with \eqref{cross-new}, and using \eqref{mu-1}, we are led to
\beq
\nn&&\fr12\fr{d}{dt}\left(\left\|\fr{\pr^{\al}P_{l\neq0}b_0}{\eps}\right\|^2_{L^2}+\|\pr^\al P_{l\neq0}\tl{u}_0\|^2_{L^2}+\mu\la \pr^\al P_{l\neq0}b_0, \Dl^{-1}\dv \pr^\al P_{l\neq0}\tl{u}_0 \ra\right)\\
\nn&&+\fr{\mu}{2}\left(\left\|\fr{\pr^\al P_{l\neq0} b_0}{\eps}\right\|^2_{L^2}+\|\nb\pr^\al P_{l\neq0}\tl{u}_0\|_{L^2}^2\right)\\
&\le&\fr{(\lm+2\mu)\mu}{2}\| \dv \pr^\al P_{l\neq0}\tl{u}_0\|_{L^2} \|\pr^\al P_{l\neq0}b_0\|_{L^2}\le\fr{\mu}{4}\left(\left\|\fr{\pr^\al P_{l\neq0} b_0}{\eps}\right\|^2_{L^2}+\|\nb\pr^\al P_{l\neq0}\tl{u}_0\|_{L^2}^2\right).
\eeq
Noticing that
\beqno
&&\mu\left(\|\nb\pr^\al P_{l\neq0}\tl{u}_0\|_{L^2}^2+\left\|\fr{\pr^\al P_{l\neq0}b_0}{\eps}\right\|^2_{L^2}\right)\nn\\
&\ge&\fr{2\mu}{3}\left(\left\|\fr{\pr^{\al}P_{l\neq0}b_0}{\eps}\right\|^2_{L^2}+\|\pr^{\al}P_{l\neq0}\tl{u}_0\|^2_{L^2}+\mu\la \pr^\al P_{l\neq0}b_0, \Dl^{-1}\dv \pr^\al P_{l\neq0}\tl{u}_0 \ra\right),
\eeqno
then \eqref{decay1} follows immediately.
\end{proof}
Recalling that the projection of $(b_0, \tl{u}_0)$ onto the zero frequencies in $z$ satisfies the system \eqref{pro-z}, it is well known (see \cite{Kob02}, for instance) that the linear operator in \eqref{pro-z} behaves like the heat semigroup, so that $(b_{00}, u^2_{00})$ decays as fast as the solution to the corresponding  heat equation  on $\mathbb{R}$ with viscous cofficient $\lm+2\mu$. For the convenience of reads, we shall give a proof of this fact in the following lemma. 
\begin{lem}\label{lem-00}
Let $k\in\mathbb{N}$. Assume that the Mach number $\eps$ and the viscous coefficients $\lm$ and $\mu$ satisfy
\be\label{mu-2}
 \lm+2\mu\le1,\ \ \mathrm{and}\ \ (\lm+2\mu)\eps\le1,
\ee
and the initial data $(({b}_{\mathrm{in}})_{00}, ({u}^2_{\mathrm{in}})_{00})\in \dot{H}^k\cap L^1$.
Then there exists a positive constant $C$,  depending on $k$, but independent of $\eps$, $\lm$ and $\mu$, such that   
\be\label{de1}
\left\|\pr^k_y\left(\fr{b_{00}}{\eps}, u^2_{00}\right)\right\|^2_{L^2}
\le C\big((\lm+2\mu)\la t\ra\big)^{-(k+\fr12)}\left\|\left(\fr{({b}_{\mathrm{in}})_{00}}{\eps},({u}^2_{\mathrm{in}})_{00}\right)\right\|^2_{\dot{H}^k\cap L^1},
\ee
and
\be\label{de2}
\left\|\pr^k_yu^3_{00}\right\|^2_{L^2}
\le C\big(\mu\la t\ra\big)^{-(k+\fr12)}\left\|({u}^3_{\mathrm{in}})_{00}\right\|^2_{\dot{H}^k\cap L^1}.
\ee
\end{lem}
\begin{proof}
We use the method in \cite{ZZZ18} to  avoid subtle analysis on the eigenvalues of the system \eqref{pro-z}. Since $u^3_{00}$ satisfies a heat equation, we only focus on the sub-system for $(b_{00}, u^2_{00})$ below. Similar to \eqref{en0-en0} and \eqref{cross-en0}, direction calculations in Fourier variable show that
\be\label{Fbu00}
\fr12\fr{d}{dt}\left(\left|\fr{\hat{b}_{00}}{\eps}\right|^2+\left|\hat{u}^2_{00}\right|^2\right)+(\lm+2\mu)\eta^2\left|\hat{u}^2_{00}\right|^2=0,
\ee
and
\be\label{Fcross}
\fr{d}{dt}\left(i\eta\bar{\hat{b}}_{00}\hat{u}^2_{00}\right)-\eta^2\left|\fr{\hat{b}_{00}}{\eps}\right|^2+\eta^2\left|\hat{u}^2_{00}\right|^2+(\lm+2\mu)i\eta^3\bar{\hat{b}}_{00}\hat{u}^2_{00}=0.
\ee
For $|\eta|\le1$,  a linear combination of \eqref{Fbu00} and \eqref{Fcross} yields
\beq\label{Fen00}
\nn&&\fr12\fr{d}{dt}\left(\left|\fr{\hat{b}_{00}}{\eps}\right|^2+\left|\hat{u}^2_{00}\right|^2-(\lm+2\mu)\left(i\eta\bar{\hat{b}}_{00}\hat{u}^2_{00}\right)\right)+\fr{(\lm+2\mu)\eta^2}{2}\left(\left|\fr{\hat{b}_{00}}{\eps}\right|^2+\left|\hat{u}^2_{00}\right|^2\right)\\
&=&\fr{(\lm+2\mu)^2\eta^2}{2}i\eta\bar{\hat{b}}_{00}\hat{u}^2_{00}\le\fr{(\lm+2\mu)\eta^2}{4}\left(\left|\fr{\hat{b}_{00}}{\eps}\right|^2+\left|\hat{u}^2_{00}\right|^2\right),
\eeq
where we have used \eqref{mu-2}. Using again the fact $|\eta|\le1$ and \eqref{mu-2}, one deduces that
\be\label{app1}
\fr12\left(\left|\fr{\hat{b}_{00}}{\eps}\right|^2+\left|\hat{u}^2_{00}\right|^2\right)\le\left|\fr{\hat{b}_{00}}{\eps}\right|^2+\left|\hat{u}^2_{00}\right|^2-(\lm+2\mu)\left(i\eta\bar{\hat{b}}_{00}\hat{u}^2_{00}\right)\le\fr32\left(\left|\fr{\hat{b}_{00}}{\eps}\right|^2+\left|\hat{u}^2_{00}\right|^2\right).
\ee
It follows from \eqref{Fen00} and \eqref{app1} that
\be\label{eta1-}
\left|\fr{\hat{b}_{00}}{\eps}\right|^2+\left|\hat{u}^2_{00}\right|^2\le Ce^{-\fr13(\lm+2\eta)\eta^2t}\left(\left|\fr{(\hat{b}_{\mathrm{in}})_{00}}{\eps}\right|^2+\left|(\hat{u}^2_{\mathrm{in}})_{00}\right|^2\right), \ \ \mathrm{if}\ \ |\eta|\le1.
\ee
For $|\eta|>1$, multiplying \eqref{Fcross} by $-\fr{\lm+2\mu}{2\eta^2}$, and adding the resulting equality to \eqref{Fbu00}, we arrive at
\beq
\nn&&\fr12\fr{d}{dt}\left(\left|\fr{\hat{b}_{00}}{\eps}\right|^2+\left|\hat{u}^2_{00}\right|^2-\fr{\lm+2\mu}{\eta}\left(i\bar{\hat{b}}_{00}\hat{u}^2_{00}\right)\right)+(\lm+2\mu)\left(\fr12\left|\fr{\hat{b}_{00}}{\eps}\right|^2+(\eta^2-\fr12)\left|\hat{u}^2_{00}\right|^2\right)\\
\nn&=&\fr{(\lm+2\mu)^2}{2}i\eta\hat{u}^2_{00}\bar{\hat{b}}_{00}\le\fr{(\lm+2\mu)}{4}\left(\left|\fr{\hat{b}_{00}}{\eps}\right|^2+\eta^2\left|\hat{u}^2_{00}\right|^2\right).
\eeq
Consequently,
\be\label{eta1+}
\left|\fr{\hat{b}_{00}}{\eps}\right|^2+\left|\hat{u}^2_{00}\right|^2\le Ce^{-\fr13(\lm+2\eta)t}\left(\left|\fr{(\hat{b}_{\mathrm{in}})_{00}}{\eps}\right|^2+\left|(\hat{u}^2_{\mathrm{in}})_{00}\right|^2\right), \ \ \mathrm{if}\ \ |\eta|>1.
\ee
Next, by virtue of Plancherel's theorem, \eqref{eta1-} and \eqref{eta1+}, we find that for all $k\in\mathbb{N}$, there holds
\beq\label{HL}
\nn\left\|\pr^k_y\left(\fr{b_{00}}{\eps}, u^2_{00}\right)\right\|^2_{L^2}&=&\int_{\mathbb{R}}\eta^{2k}\left(\left|\fr{\hat{b}_{00}}{\eps}\right|^2+\left|\hat{u}^2_{00}\right|^2\right)d\eta\\
\nn&\le&C\int_{|\eta|\le1}\eta^{2k}e^{-\fr13(\lm+2\mu)\eta^2t}d\eta\left(\left\|\fr{(\hat{b}_{\mathrm{in}})_{00}}{\eps}\right\|^2_{L^\infty}+\left\|(\hat{u}^2_{\mathrm{in}})_{00}\right\|^2_{L^\infty}\right)\\
\nn&&+Ce^{-\fr13(\lm+2\mu)t}\int_{|\eta|\ge1}\eta^{2k}\left(\left|\fr{(\hat{b}_{\mathrm{in}})_{00}}{\eps}\right|^2+\left|(\hat{u}^2_{\mathrm{in}})_{00}\right|^2\right)d\eta\\
&\le&Ce^{\fr13(\lm+2\mu)}\left(((\lm+2\mu)(1+t))^{-(k+\fr12)}+e^{-\fr13(\lm+2\mu)(1+t)}\right)\left\|\left(\fr{({b}_{\mathrm{in}})_{00}}{\eps},({u}^2_{\mathrm{in}})_{00}\right)\right\|^2_{\dot{H}^k\cap L^1}.
\eeq
Therefore, in view of \eqref{mu-2}, the inequality \eqref{de1} holds, and \eqref{de2} can be obtained in the same manner as \eqref{HL}. This completes the proof of Lemma \ref{lem-00}.
\end{proof}
\subsection{The lift-up effect} Now we are in a position to investigate $u^1_0$. Noting that $u^1_0$ satisfies $\eqref{Li-0}_3$, which can be solved explicitly
\be\label{u10}
u^1_0(t)=e^{\mu t\Dl}(u^1_{\mathrm{in}})_0-\int_0^te^{\mu (t-s)\Dl}u^2_0(s)ds.
\ee
We would like to remark that in the corresponding 3D incompressible case, $u^2_0(s)=e^{\mu s\Dl}(u^2_0)_{\mathrm{in}}$, and the formula \eqref{u10} reduces to 
\[
u^1_0(t)=e^{\mu t\Dl}\left((u^1_{\mathrm{in}})_0-t(u^2_{\mathrm{in}})_0\right).
\]
The linear in time growth predicted by this formula for $t\lesssim \mu^{-1}$ is known as the lift-up effect. Moreover, in that case the divergence free condition implies that $u^2_0=P_{l\neq0}u^2_0$, and the following estimate for $u^1_0$ is available
\be\label{incom-u10}
\|u^1_0\|^2_{L^\infty L^2}+\mu\|\nb u^1_0\|^2_{L^2L2}\le \mu^{-1}\|(u_0)_{\mathrm{in}}\|^2_{L^2}.
\ee
Nevertheless, in this paper the spectral gap  made via the divergence free condition does not hold anymore, and it is more difficult to  get a uniform bound for $u^1_0$ like \eqref{incom-u10}. Our strategy is to get  faster decay of $u^2_0$ at the price of more restriction on  the initial data. 
\begin{prop}\label{prop-liftup}
Assume that the Mach number $\eps$ and the viscous coefficients $\lm$ and $\mu$ satisfy \eqref{mu-2}, and the projection of the initial data onto the zero frequencies in $x$ and $z$ satisfy $(({b}_{\mathrm{in}})_{00}, ({u}^2_{\mathrm{in}})_{00})\in L^2\cap L^1$.
Then there exists a positive constant $C$,  independent of $\eps$, $\lm$ and $\mu$, such that
\beq\label{liftup}
\nn \|u^1_0(t)\|_{L^2}+\mu^\fr12\|\nb u^1_0\|_{L^2L^2}&\leq& \|(u^1_{\mathrm{in}})_0\|_{L^2}+C\mu^{-1}\left(\left\|\fr{P_{l\neq0}(b_{\mathrm{in}})_0}{\eps}\right\|_{L^2}+\|P_{l\neq0}(\tl{u}_{\mathrm{in}})_0\|_{L^2}\right)\\
&&+C(\lm+2\mu)^{-\fr14}\la t\ra^{\fr34}\left\|\left(\fr{({b}_{\mathrm{in}})_{00}}{\eps},({u}^2_{\mathrm{in}})_{00}\right)\right\|_{L^2\cap L^1}.
\eeq
\end{prop}
\begin{proof}
From \eqref{u10}, \eqref{decay1} and \eqref{de1}, we find that
\beq
\nn\|u^1_0(t)\|_{L^2}&\le&\|e^{\mu t\Dl}(u^1_{\mathrm{in}})_0\|_{L^2}+\int_0^t\|e^{\mu (t-s)\Dl}P_{l\neq0}u^2_0(s)\|_{L^2}+\|u^2_{00}(s)\|_{L^2}ds\\
\nn&\le&\|(u^1_{\mathrm{in}})_0\|_{L^2}+\int_0^te^{-\mu (t-s)}ds\|P_{l\neq0}u^2_0(s)\|_{L^\infty L^2}\\
\nn&&+C\int_0^t\big((\lm+2\mu)(1+s)\big)^{-\fr14}ds\left\|\left(\fr{({b}_{\mathrm{in}})_{00}}{\eps},({u}^2_{\mathrm{in}})_{00}\right)\right\|_{L^2\cap L^1}\\
\nn&\le&\|(u^1_{\mathrm{in}})_0\|_{L^2}+C\mu^{-1}\left(\left\|\fr{P_{l\neq0}(b_{\mathrm{in}})_0}{\eps}\right\|_{L^2}+\|P_{l\neq0}(\tl{u}_{\mathrm{in}})_0\|_{L^2}\right)\\
\nn&&+C(\lm+2\mu)^{-\fr14}\la t\ra^{\fr34}\left\|\left(\fr{({b}_{\mathrm{in}})_{00}}{\eps},({u}^2_{\mathrm{in}})_{00}\right)\right\|_{L^2\cap L^1},
\eeq
and
\beno
\mu^\fr12\|\nb u^1_0\|_{L^2L^2}\le \mu^\fr12\left\|\nb\left(e^{\mu t\Dl}(u^1_{\mathrm{in}})_0\right)\right\|_{L^2L^2}+\mu^\fr12\left\|\int_0^\tau\nb \left(e^{\mu(\tau-s)\Dl}u^2_0(s)\right)ds\right\|_{L^2L^2}=:I+II.
\eeno
Using Plancherel's theorem, we have
\beqno
I=\left(\sum_{l\in\mathbb{Z}}\int_{\mathbb{R}}\int_0^t\mu(\eta^2+l^2)e^{-2\mu s(\eta^2+l^2)}ds \left|(\hat{u}^1_{\mathrm{in}})_0\right|^2d\eta\right)^\fr12\le\fr{1}{\sqrt{2}}\|(u^1_{\mathrm{in}})_0\|_{L^2}.
\eeqno
In view of Minkowski's inequality,  and employing  \eqref{decay1} and \eqref{de1} again, one deduces that
\beq
\nn II&\le&\left\{\int_0^t\left[\int_0^\tau\left(\sum_{l\in\mathbb{Z}}\int_{\mathbb{R}}\mu(\eta^2+l^2)e^{-2\mu(\tau-s)(\eta^2+l^2)}\left|\hat{u}^2_0(s)\right|^2d\eta\right)^{\fr12}ds\right]^2d\tau\right\}^\fr12\\
\nn&\le&\int_0^t\left(\int_s^t\sum_{l\in\mathbb{Z}}\int_{\mathbb{R}}\mu(\eta^2+l^2)e^{-2\mu(\tau-s)(\eta^2+l^2)}\left|\hat{u}^2_0(s)\right|^2d\eta d\tau\right)^{\fr12}ds\\
\nn&\le&\fr{1}{\sqrt{2}}\int_0^t\|u^2_0(s)\|_{L^2}ds\\
\nn&\le& C\mu^{-1}\left(\left\|\fr{P_{l\neq0}(b_{\mathrm{in}})_0}{\eps}\right\|_{L^2}+\|P_{l\neq0}(\tl{u}_{\mathrm{in}})_0\|_{L^2}\right)\\
\nn&&+(\lm+2\mu)^{-\fr14}\la t\ra^{\fr34}\left\|\left(\fr{({b}_{\mathrm{in}})_{00}}{\eps},({u}^2_{\mathrm{in}})_{00}\right)\right\|_{L^2\cap L^1}.
\eeq
We complete the proof of Proposition \ref{prop-liftup}.
\end{proof}
\section{Estimates of non-zero mode}
The purpose of this section is to obtain the enhanced dissipation of the non-zero mode of $(B^i, D^i, W^i)$ for $i=1, 2, 3$. More precisely, we will establish the following proposition.
\begin{prop}\label{prop-endis}
Let $s\ge0$. Assume that the Mach number $\eps$ and the viscous coefficients $\mu$ and $\lm$ satisfy \eqref{visco} and \eqref{mu}. Then there exists a constant $C$, depending on $\eps$, but independent of  $\lm$ and $\mu$, such that
\beq\label{enhan13}
\nn&&\sum_{i=1,3}\left(\left\|m^{-\fr34}\sqrt{-\Dl_L}\frac{B^i_\neq}{\eps}\right\|^2_{H^s}+\|m^{-\fr34}D^i_\neq\|^2_{H^s}\right)+\|m^{-\fr34}W^2_\neq\|^2_{H^s}+\|m^{-1}W^3_\neq\|^2_{H^s}\\
&\le&C e^{-\fr{\mu^\fr13t}{22}}\Bigg(\left\|\fr{(\nb{\nb_{x,z} b_{\mathrm{in}}})_{\neq}}{\eps}\right\|^2_{H^s}+\left\|(\nb_{x,z} \dv u_{\mathrm{in}})_{\neq}\right\|^2_{H^s}+\|(w^2_{\mathrm{in}})_\neq+\pr_xb_{\mathrm{in}}\|^2_{H^s}+\|(w^3_{\mathrm{in}})_\neq\|^2_{H^s}\Bigg),
\eeq
and
\beq\label{enhan2}
\nn&&\left\|m^{-\fr34}\sqrt{-\Dl_L}\frac{B^2_\neq}{\eps}\right\|^2_{H^s}+\|m^{-\fr34}D^2_\neq\|^2_{H^s}+\|m^{-1}W^1_\neq\|^2_{H^s}\\
\nn&\le&C e^{-\fr{\mu^\fr13t}{44}}\Bigg\{\mu^{-\fr23}\Bigg(\left\|\fr{(\nb{\nb_{x,z} b_{\mathrm{in}}})_{\neq}}{\eps}\right\|^2_{H^s}+\left\|(\nb_{x,z} \dv u_{\mathrm{in}})_{\neq}\right\|^2_{H^s}+\|(w^2_{\mathrm{in}})_\neq+\pr_xb_{\mathrm{in}}\|^2_{H^s}+\|(w^3_{\mathrm{in}})_\neq\|^2_{H^s}\Bigg)\\
&&+\Bigg(\left\|\fr{(\nb{\pr_{y} b_{\mathrm{in}}})_{\neq}}{\eps}\right\|^2_{H^s}+\left\|(\pr_{y} \dv u_{\mathrm{in}})_{\neq}\right\|^2_{H^s}+\|(w^1_{\mathrm{in}})_\neq-(\pr_yb_{\mathrm{in}})_\neq\|^2_{H^s}\Bigg)\Bigg\}.
\eeq
\end{prop}
\begin{rem}\label{rem4.1}
Since the system \eqref{Lx} of $({\bf b}^1, {\bf d}^1, {\bf w}^2)$ is self-closed, if one just focus on the system \eqref{Lx}, the proof in subsection \ref{subsec-13}  below  actually gives a more accurate estimate on $({\bf b}^1, {\bf d}^1, {\bf w}^2)$:
\beq\label{enhan1}
\nn&&\left(\left\|m^{-\fr34}\sqrt{-\Dl_L}\frac{B^1_\neq}{\eps}\right\|^2_{H^s}+\|m^{-\fr34}D^1_\neq\|^2_{H^s}\right)+\|m^{-\fr34}W^2_\neq\|^2_{H^s}\\
&\le&C e^{-\fr{\mu^\fr13t}{22}}\left(\left\|\fr{\nb{\pr_{x} b_{\mathrm{in}}}}{\eps}\right\|^2_{H^s}+\left\|\pr_{x} \dv u_{\mathrm{in}}\right\|^2_{H^s}+\|(w^2_{\mathrm{in}})_\neq+\pr_xb_{\mathrm{in}}\|^2_{H^s}\right).
\eeq
Moreover,  one can see from system \eqref{Lz} that $({\bf b}^3, {\bf d}^3)$ and $w^3$ are linearly independent of each other, but both depend on ${\bf b}^1$ and ${\bf w}^2$. Thus $({\bf b}^1, {\bf d}^1, {\bf w}^2, {\bf b}^3, {\bf d}^3)$ and $({\bf b}^1, {\bf d}^1, {\bf w}^2, w^3)$ can be bounded separately. As for  the system \eqref{Ly}, it is  clear that $({\bf b}^2, {\bf d}^2, {\bf w}^1)$ does not depend on $({\bf b}^3, {\bf d}^3, w^3)$. Therefore, in order to get the bounds for $(B^2, D^2, W^1)$,  $(B^3, D^3, W^3)$ is not necessarily involved in subsection \ref{subsec-4.2}. Accordingly, the derivatives of the initial data on the $z$-direction and $w^3_{\mathrm{in}}$ can be removed on the right hand side of \eqref{enhan2}.
\end{rem}

Before proceeding any further, let us denote
\be\label{wei}
\begin{split}
M:=&\la k, \eta, l\ra^sm^{-\fr34}m_1^{-1}m_2^{-1}m_3^{-1}{\bf 1}_{k\neq0},\\
{M}_{1}:=&\la k, \eta, l\ra^sm^{-1}m_1^{-1}m_2^{-1}m_3^{-1}{\bf 1}_{k\neq0},\\
h^2:=&cM^2,\\
g^2:=&\fr14\fr{\pr_tm}{m}M^2,
\end{split}
\ee
with  the constant $N$ appearing in the definitions of $m_1$ and $m_3$ (see Definition \ref{def-adm}) and the constant $c$ above to be determined below. The power $\fr34$ of $m^{-1}$ in $M$ and the factor $\fr14$ in $g^2$  are closely related to each other, which was first observed in \cite{ADM20} for the 2D inviscid flow. The proof of Proposition \ref{prop-endis} will be achieved in the following two subsections.
\subsection{Estimates of $(B^1, D^1, W^2)$ and $(B^3, D^3, W^3)$}\label{subsec-13}
Define
\be\label{en1}
E_s^1(t):=\sum_{i=1,3}\left(\left\|M\sqrt{-\Dl_L}\frac{B^i}{\eps}\right\|^2_{L^2}+\left\|\sqrt{\fr{\pr_tm}{m}}MB^i\right\|_{L^2}^2+\|MD^i\|^2_{L^2}\right)+\|MW^2\|^2_{L^2}+\|M_1W^3\|^2_{L^2},
\ee
and
\be\label{var-en1}
\mathcal{E}_s^1(t):=E_s^1(t)+2\sum_{i=1,3}\left(\la gB^i, gD^i\ra-\la hB^i, hD^i\ra\right).
\ee
We would like to remark that if the Mach number $\eps\in(0, 1]$, the second term in \eqref{en1}, first introduced by Antonelli, Dolce and Marcati in \cite{ADM21}, is not needed.\par
\noindent{\bf Step (I): Estimates of the the functional $E_s^1(t)$.}
First of all from \eqref{Lx'} and \eqref{Lz'}, we find that for $i=1, 3$, there hold
\beq\label{evo-nbB}
\fr12\fr{d}{dt}\left\|M\sqrt{-\Dl_L}\frac{B^i}{\eps}\right\|^2_{L^2}
\nn&=&-\fr34\left\|\sqrt{\fr{\pr_tm}{m}}M\sqrt{-\Dl_L}\frac{B^i}{\eps}\right\|^2_{L^2}-\sum_{j=1,2,3}\left\|\sqrt{\fr{\pr_tm_j}{m_j}}M\sqrt{-\Dl_L}\frac{B^i}{\eps}\right\|^2_{L^2}\\
&&+\fr12\sum_{k,l}\int_\mathbb{R}M^2p'\left|\frac{\hat{B^i}}{\eps}\right|^2d\eta+\frac{1}{\eps^2}\left\la M\Dl_LB^i, MD^i\right\ra,
\eeq
and
\beq
\fr{1}{2}\fr{d}{dt}\|MD^i\|^2_{L^2}
\nn&=&-\fr34\left\|\sqrt{\fr{\pr_tm}{m}}MD^i\right\|^2_{L^2}-\sum_{j=1,2,3}\left\|\sqrt{\fr{\pr_tm_j}{m_j}}MD^i\right\|^2_{L^2}\\
\nn&&-(\lm+2\mu)\left\|M\sqrt{-\Dl_L}D^i\right\|^2_{L^2}+\sum_{k, l}\int_\R M^2\fr{p'}{p}|\hat{D}^i|^2d\eta\\
&&-\frac{1}{\eps^2}\left\la M\Dl_LB^i, MD^i\right\ra-\begin{cases}
\displaystyle2\sum_{k, l}\int_\R M^2\fr{k^2}{p}(\hat{W}^2-\hat{B}^1)\bar{\hat{D}}^1d\eta, \quad{\rm if}\quad i=1,\\[2mm]
\displaystyle2\sum_{k, l}\int_\R M^2\fr{kl}{p}(\hat{W}^2-\hat{B}^1)\bar{\hat{D}}^3d\eta,\quad{\rm if}\quad i=3.
\end{cases}
\eeq
Similar to \eqref{evo-nbB}, we get the evolution of $\left\|\sqrt{\fr{\pr_tm}{m}}MB^i\right\|_{L^2}^2$,
\beq\label{E-corre}
\fr12\fr{d}{dt}\left\|\sqrt{\fr{\pr_tm}{m}}MB^i\right\|^2_{L^2}\nn&=&-\fr34\left\|{\fr{\pr_tm}{m}}M{B^i}\right\|^2_{L^2}-\sum_{j=1,2,3}\left\|\sqrt{\fr{\pr_tm}{m}\fr{\pr_tm_j}{m_j}}M{B^i}\right\|^2_{L^2}\\
&&+\fr12\sum_{k,l}\int_{\mathbb{R}}\pr_t\left(\fr{\pr_tm}{m}\right)M^2|\hat{B}^i|^2d\eta-\sum_{k,l}\int_{\mathbb{R}}\fr{\pr_tm}{m}M^2\hat{B}^i\bar{\hat{D}}^id\eta.
\eeq
The evolutions of the last two terms in the energy functional $E_s^1(t)$ are given as follows:
\beq\label{e-W2}
\fr{1}{2}\fr{d}{dt}\|MW^2\|^2_{L^2}
\nn&=&-\fr34\left\|\sqrt{\fr{\pr_tm}{m}}MW^2\right\|^2_{L^2}-\sum_{j=1,2,3}\left\|\sqrt{\fr{\pr_tm_j}{m_j}}MW^2\right\|^2_{L^2}\\
&&-\mu\left\|M\sqrt{-\Dl_L}W^2\right\|^2_{L^2}+\mu\left\la M\sqrt{-\Dl_L}B^1, M\sqrt{-\Dl_L}W^2\right\ra,
\eeq
and
\beq\label{evo-W3}
\fr{1}{2}\fr{d}{dt}\|M_1W^3\|^2_{L^2}
\nn&=&-\left\|\sqrt{\fr{\pr_tm}{m}}M_1W^3\right\|^2_{L^2}-\sum_{j=1,2,3}\left\|\sqrt{\fr{\pr_tm_j}{m_j}}M_1W^3\right\|^2_{L^2}-\mu\left\|M_1\sqrt{-\Dl_L}W^3\right\|^2_{L^2}\\
&&+\sum_{k, l}\int_\R M_1^2\fr{p'}{p}|\hat{W}^3|^2d\eta+2\sum_{k, l}\int_\R M_1^2\fr{kl}{p}(\hat{W}^2-\hat{B}^1)\bar{\hat{W}}^3d\eta.
\eeq
From \eqref{evo-nbB}--\eqref{evo-W3}, one deduces the evolution of the energy functional $E_s^1(t)$,
\beq\label{e1}
\nn&&\fr{1}{2}\fr{d}{dt}E_s^1(t)+(\lm+2\mu)\sum_{i=1,3}\left\|M\sqrt{-\Dl_L}D^i\right\|^2_{L^2}+\mu\left\|M\sqrt{-\Dl_L}W^2\right\|^2_{L^2}+\mu\left\|M_1\sqrt{-\Dl_L}W^3\right\|^2_{L^2}\\
\nn&&+\fr34\sum_{i=1,3}\left\|{\fr{\pr_tm}{m}}M{B^i}\right\|^2_{L^2}+\sum_{i=1,3}\sum_{j=1,2,3}\left\|\sqrt{\fr{\pr_tm}{m}\fr{\pr_tm_j}{m_j}}M{B^i}\right\|^2_{L^2}\\
\nn&&+\fr34\sum_{i=1,3}\left\|\sqrt{\fr{\pr_tm}{m}}M\sqrt{-\Dl_L}\fr{B^i}{\eps}\right\|^2_{L^2}+\sum_{i=1,3}\sum_{j=1,2,3}\left\|\sqrt{\fr{\pr_tm_j}{m_j}}M\sqrt{-\Dl_L}\fr{B^i}{\eps}\right\|^2_{L^2}\\
\nn&&+\fr34\sum_{i=1,3}\left\|\sqrt{\fr{\pr_tm}{m}}M{D^i}\right\|^2_{L^2}+\sum_{i=1,3}\sum_{j=1,2,3}\left\|\sqrt{\fr{\pr_tm_j}{m_j}}M{D}^i\right\|^2_{L^2}\\
\nn&&+\fr34\left\|\sqrt{\fr{\pr_tm}{m}}M{W^2}\right\|^2_{L^2}+\sum_{j=1,2,3}\left\|\sqrt{\fr{\pr_tm_j}{m_j}}MW^2\right\|^2_{L^2}\\
\nn&&+\left\|\sqrt{\fr{\pr_tm}{m}}M_1W^3\right\|^2_{L^2}+\sum_{j=1,2,3}\left\|\sqrt{\fr{\pr_tm_j}{m_j}}M_1W^3\right\|^2_{L^2}\\
&=&\fr12\sum_{i=1,3}\sum_{k, l}\int_\mathbb{R}M^2p'\left|\fr{\hat{B}^i}{\eps}\right|^2d\eta+\sum_{i=1,3}\sum_{k,l}\int_\R M^2\fr{p'}{p}|\hat{D}^i|^2d\eta+\sum_{k, l}\int_\R M_1^2\fr{p'}{p}|\hat{W}^3|^2d\eta\\
\nn&&-2\sum_{k, l}\int_\R M^2\fr{k^2}{p}(\hat{W}^2-\hat{B}^1)\bar{\hat{D}}^1d\eta-2\sum_{k, l}\int_\R M^2\fr{kl}{p}(\hat{W}^2-\hat{B}^1)\bar{\hat{D}}^3d\eta\\
\nn&&+2\sum_{k, l}\int_\R M_1^2\fr{kl}{p}(\hat{W}^2-\hat{B}^1)\bar{\hat{W}}^3d\eta+\mu\left\la M\sqrt{-\Dl_L}B^1, M\sqrt{-\Dl_L}W^2\right\ra\\
\nn&&+\fr12\sum_{i=1,3}\sum_{k,l}\int_{\mathbb{R}}\pr_t\left(\fr{\pr_tm}{m}\right)M^2|\hat{B}^i|^2d\eta-\sum_{i=1,3}\sum_{k,l}\int_{\mathbb{R}}\fr{\pr_tm}{m}M^2\hat{B}^i\bar{\hat{D}}^id\eta.
\eeq
Now we go to bound the right hand side of \eqref{e1}. Firstly, we infer from \eqref{vis-dom}--\eqref{p'p} in Remark \ref{rem-m} that
\beq\label{role-m}
\nn&&\fr12\sum_{i=1,3}\sum_{k, l}\int_\mathbb{R}M^2p'\left|\fr{\hat{B}^i}{\eps}\right|^2d\eta+\sum_{i=1,3}\sum_{k,l}\int_\R M^2\fr{p'}{p}|\hat{D}^i|^2d\eta+\sum_{k, l}\int_\R M_1^2\fr{p'}{p}|\hat{W}^3|^2d\eta\\
&\le&\sum_{i=1,3}\left(\fr12\left\|\sqrt{\fr{\pr_tm}{m}}M\sqrt{-\Dl_L}\fr{B^i}{\eps}\right\|^2_{L^2}+\left\|\sqrt{\fr{\pr_tm}{m}}M{D}^i\right\|^2_{L^2}\right)+\left\|\sqrt{\fr{\pr_tm}{m}}M_1{W}^3\right\|^2_{L^2}\\
&&\nn+{\fr{\mu^\fr13}{32}\left(\sum_{i=1,3}\left(\fr12\left\|M\sqrt{-\Dl_L}\fr{B^i}{\eps}\right\|^2_{L^2}+\left\|MD^i\right\|^2_{L^2}\right)+\left\|M_1W^3\right\|^2_{L^2}\right)}.
\eeq
Recalling the definitions of $m^1$ and $m^3$, and using the fact $M_1\le M$, we are led to
\beq
\nn&&-2\sum_{k, l}\int_\R M^2\fr{k^2}{p}(\hat{W}^2-\hat{B}^1)\bar{\hat{D}}^1d\eta-2\sum_{k, l}\int_\R M^2\fr{kl}{p}(\hat{W}^2-\hat{B}^1)\bar{\hat{D}}^3d\eta\\
\nn&&+2\sum_{k, l}\int_\R M_1^2\fr{kl}{p}(\hat{W}^2-\hat{B}^1)\bar{\hat{W}}^3d\eta\\
\nn&\le&\fr{2}{N}\left(\left\|\sqrt{\fr{\pr_tm_1}{m_1}}MW^2\right\|_{L^2}+\eps\left\|\sqrt{\fr{\pr_tm_1}{m_1}}M\fr{B^1}{\eps}\right\|_{L^2}\right)\left\|\sqrt{\fr{\pr_tm_1}{m_1}}MD^1\right\|_{L^2}\\
\nn&&+\fr{2}{N}\left(\left\|\sqrt{\fr{\pr_tm_3}{m_3}}MW^2\right\|_{L^2}+\eps\left\|\sqrt{\fr{\pr_tm_3}{m_3}}M\fr{B^1}{\eps}\right\|_{L^2}\right)\left\|\sqrt{\fr{\pr_tm_3}{m_3}}MD^3\right\|_{L^2}\\
\nn&&+\fr{2}{N}\left(\left\|\sqrt{\fr{\pr_tm_3}{m_3}}M_1W^2\right\|_{L^2}+\eps\left\|\sqrt{\fr{\pr_tm_3}{m_3}}M_1\fr{B^1}{\eps}\right\|_{L^2}\right)\left\|\sqrt{\fr{\pr_tm_3}{m_3}}M_1W^3\right\|_{L^2}\\
&\le&\fr{2}{N}\sum_{j=1,3}\left\|\sqrt{\fr{\pr_tm_j}{m_j}}MW^2\right\|_{L^2}^2+\fr{2\eps}{N}\sum_{j=1,3}\left\|\sqrt{\fr{\pr_tm_j}{m_j}}M\fr{B^1}{\eps}\right\|_{L^2}^2\\
\nn&&+\fr{1+\eps}{N}\left(\sum_{j=1,3}\left\|\sqrt{\fr{\pr_tm_j}{m_j}}MD^j\right\|_{L^2}^2+\left\|\sqrt{\fr{\pr_tm_3}{m_3}}M_1W^3\right\|_{L^2}^2\right).
\eeq
Next, by the definition of $m$ and $m_1$, thanks to the fact $\fr{|p'|}{p}\le\fr{2|k|}{\sqrt{p}}$, one deduces that
\be\label{wtw}
0\le\fr{\pr_tm}{m}\le\fr{|p'|}{p}\le1,
\ee

\be\label{wtwp}
0\le\fr{\pr_tm}{m\sqrt{p}}\le\fr{|p'|}{p^{\fr32}}\le\fr{2|k|}{p}\le\fr{2}{N}\fr{\pr_tm_1}{m_1},
\ee
and
\be\label{ptp'}
\begin{split}
\left|\pr_t\left(\fr{\pr_tm}{m}\right)\right|\le\fr{2k^2}{p}+\fr{|p'|^2}{p^2}\le\fr{6k^2}{p}\le\fr{6}{N}\fr{\pr_tm_1}{m_1}.
\end{split}
\ee
Consequently,
\be\label{Bi1}
\fr12\sum_{i=1,3}\sum_{k,l}\int_{\mathbb{R}}\pr_t\left(\fr{\pr_tm}{m}\right)M^2|\hat{B}^i|^2d\eta\le\fr{3\eps^2}{N}\sum_{i=1,3}\left\|\sqrt{\fr{\pr_tm_1}{m_1}}M\sqrt{-\Dl_L}\fr{B^i}{\eps}\right\|^2_{L^2},
\ee
and
\beq\label{Bi2}
\nn-\sum_{i=1,3}\sum_{k,l}\int_{\mathbb{R}}\fr{\pr_tm}{m}M^2\hat{B}^i\bar{\hat{D}}^id\eta&\le&\fr{2\eps}{N}\sum_{i=1,3}\left\|\sqrt{\fr{\pr_tm_1}{m_1}}M\sqrt{-\Dl_L}\fr{B^i}{\eps}\right\|_{L^2}\left\|\sqrt{\fr{\pr_tm_1}{m_1}}M{D^i}\right\|_{L^2}\\
&\le&\fr{\eps}{N}\sum_{i=1,3}\left(\left\|\sqrt{\fr{\pr_tm_1}{m_1}}M\sqrt{-\Dl_L}\fr{B^i}{\eps}\right\|_{L^2}^2+\left\|\sqrt{\fr{\pr_tm_1}{m_1}}M{D^i}\right\|_{L^2}^2\right).
\eeq
Finally, by virtue of the Cauchy-Schwarz inequality,  we have
\be
{\mu\left\la M\sqrt{-\Dl_L}B^1, M\sqrt{-\Dl_L}W^2\right\ra\le\fr{\mu}{2}\|M\sqrt{-\Dl_L}W^2\|^2_{L^2}+\fr{\mu\eps^2}{2}\left\|M\sqrt{-\Dl_L}\fr{B^1}{\eps}\right\|^2_{L^2}}.
\ee
Substituting the above  inequalities into \eqref{e1} yields
\beq\label{Es}
\nn&&\fr{1}{2}\fr{d}{dt}E_s^1(t)+(\lm+2\mu)\sum_{i=1,3}\left\|M\sqrt{-\Dl_L}D^i\right\|^2_{L^2}+\fr{\mu}{2}\left\|M\sqrt{-\Dl_L}W^2\right\|^2_{L^2}\\
\nn&&+\mu\left\|M_1\sqrt{-\Dl_L}W^3\right\|^2_{L^2}+\fr14\sum_{i=1,3}\left\|\sqrt{\fr{\pr_tm}{m}}M\sqrt{-\Dl_L}\fr{B^i}{\eps}\right\|^2_{L^2}+\fr34\left\|\sqrt{\fr{\pr_tm}{m}}M{W^2}\right\|^2_{L^2}\\
\nn&&+\fr34\sum_{i=1,3}\left\|{\fr{\pr_tm}{m}}M{B^i}\right\|^2_{L^2}+\sum_{i=1,3}\sum_{j=1,2,3}\left\|\sqrt{\fr{\pr_tm}{m}\fr{\pr_tm_j}{m_j}}M{B^i}\right\|^2_{L^2}\\
\nn&&+\sum_{i=1,3}\left(\left\|\sqrt{\fr{\pr_tm_2}{m_2}}M\sqrt{-\Dl_L}\fr{B^i}{\eps}\right\|^2_{L^2}+\left\|\sqrt{\fr{\pr_tm_2}{m_2}}M{D}^i\right\|^2_{L^2}\right)\\
&&+\left\|\sqrt{\fr{\pr_tm_2}{m_2}}MW^2\right\|^2_{L^2}+\left\|\sqrt{\fr{\pr_tm_2}{m_2}}M_1W^3\right\|^2_{L^2}\\
\nn&&+(1-\fr{\eps_1}{N})\sum_{i=1,3}\sum_{j=1,3}\left(\left\|\sqrt{\fr{\pr_tm_j}{m_j}}M\sqrt{-\Dl_L}\fr{B^i}{\eps}\right\|^2_{L^2}+\left\|\sqrt{\fr{\pr_tm_j}{m_j}}M{D}^i\right\|^2_{L^2}\right)\\
\nn&&+(1-\fr{\eps_1}{N})\sum_{j=1,3}\left(\left\|\sqrt{\fr{\pr_tm_j}{m_j}}MW^2\right\|^2_{L^2}+\left\|\sqrt{\fr{\pr_tm_j}{m_j}}M_1W^3\right\|^2_{L^2}\right)\\
\nn&\le&{\fr{\mu^\fr13}{32}\left(\sum_{i=1,3}\left(\fr12\left\|M\sqrt{-\Dl_L}\fr{B^i}{\eps}\right\|^2_{L^2}+\left\|MD^i\right\|^2_{L^2}\right)+\left\|M_1W^3\right\|^2_{L^2}\right)}\\
\nn&&{+\fr{\mu\eps^2}{2}\left\|M\sqrt{-\Dl_L}\fr{B^1}{\eps}\right\|^2_{L^2}}+\fr14\sum_{i=1,3}\left\|\sqrt{\fr{\pr_tm}{m}}M{D^i}\right\|^2_{L^2},
\eeq
where 
\be\label{eps1}
\eps_1:=\max\left\{2, 1+2\eps, 3\eps+3\eps^2\right\}.
\ee
\noindent{\bf Step (II): Estimates of the modified energy functional $\mathcal{E}_s^1(t)$.}
To this end, we turn to derive the evolutions of the cross terms appearing in the modified energy functional $\mathcal{E}_s^1(t)$. Note  first that
\be\label{h2t}
\pr_t(h^2)=-2c\sum_{j=1,2,3}\fr{\pr_tm_j}{m_j}M^2-\fr{3c}{2}\fr{\pr_tm}{m}M^2,
\ee
and
\be\label{g2t}
\begin{split}
\pr_t(g^2)=&-\fr12\sum_{j=1,2,3}\fr{\pr_tm}{m}\fr{\pr_tm_j}{m_j}M^2-\fr38\left(\fr{\pr_tm}{m}\right)^2M^2+\fr14\pr_t\left(\fr{\pr_tm}{m}\right)M^2.
\end{split}
\ee
Then we have
\beq\label{in-h'}
\fr{d}{dt}\la hB^i, hD^i\ra\nn&=&c\left\|M\sqrt{-\Dl_L}\fr{B^i}{\eps}\right\|^2_{L^2}-c\|MD^i\|^2_{L^2}\\
\nn&&-c(\lm+2\mu)\sum_{k,l}\int_\R pM^2\hat{B}^i\bar{\hat{D}}^id\eta+c\sum_{k,l}\int_\R M^2\fr{p'}{p}\hat{B}^i\bar{\hat{D}}^id\eta\\
\nn&&-2c\sum_{j=1,2,3}\sum_{k,l}\int_\R M^2\fr{\pr_tm_j}{m_j}\hat{B}^i\bar{\hat{D}}^id\eta-\fr{3c}{2}\sum_{k,l}\int_\R M^2\fr{\pr_tm}{m}\hat{B}^i\bar{\hat{D}}^id\eta\\
&&-\begin{cases}
\displaystyle2c\sum_{k,l}\int_\R M^2\fr{k^2}{p}\hat{B}^1(\bar{\hat{W}}^2-\bar{\hat{B}}^1)d\eta,\quad{\rm if}\quad i=1,\\[2mm]
\displaystyle2c\sum_{k,l}\int_\R M^2\fr{kl}{p}\hat{B}^3(\bar{\hat{W}}^2-\bar{\hat{B}}^1)d\eta, \quad{\rm if}\quad i=3,
\end{cases}
\eeq
and
\beq\label{in-g'}
\nn&&\fr{d}{dt}\la gB^i, gD^i\ra\\
\nn&=&\fr14\left\|\sqrt{\fr{\pr_tm}{m}}M\sqrt{-\Dl_L}\fr{B^i}{\eps}\right\|^2_{L^2}-\fr14\left\|\sqrt{\fr{\pr_tm}{m}}MD^i\right\|^2_{L^2}\\
\nn&&-\fr{\lm+2\mu}{4}\sum_{k,l}\int_\R p\fr{\pr_tm}{m}M^2\hat{B}^i\bar{\hat{D}}^id\eta+\fr14\sum_{k,l}\int_\R \fr{\pr_tm}{m}M^2\fr{p'}{p}\hat{B}^i\bar{\hat{D}}^id\eta\\
\nn&&-\fr12\sum_{j=1,2,3}\sum_{k,l}\int_\R \fr{\pr_tm}{m}\fr{\pr_tm_j}{m_j}M^2\hat{B}^i\bar{\hat{D}}^id\eta-\fr38\sum_{k,l}\int_\R \left(\fr{\pr_tm}{m}\right)^2M^2\hat{B}^i\bar{\hat{D}}^id\eta\\
&&+\fr14\sum_{k,l}\int_\R\pr_t\left(\fr{\pr_tm}{m}\right)M^2\hat{B}^i\bar{\hat{D}}^id\eta-\begin{cases}
\displaystyle\fr12\sum_{k,l}\int_\R \fr{\pr_tm}{m}M^2\fr{k^2}{p}\hat{B}^1(\bar{\hat{W}}^2-\bar{\hat{B}}^1)d\eta,\quad{\rm if}\quad i=1,\\[2mm]
\displaystyle\fr12\sum_{k,l}\int_\R \fr{\pr_tm}{m}M^2\fr{kl}{p}\hat{B}^3(\bar{\hat{W}}^2-\bar{\hat{B}}^1)d\eta, \quad{\rm if}\quad i=3.
\end{cases}
\eeq
Using the Cauchy-Schwarz inequality and \eqref{wtw}, we obtain
\beq\label{role-mj}
2c\sum_{j=1,2,3}\sum_{k,l}\int_\R M^2\fr{\pr_tm_j}{m_j}\hat{B}^i\bar{\hat{D}}^id
\nn&\le&2c\sum_{j=1,2,3}\left\|\sqrt{\fr{\pr_tm_j}{m_j}}MB^i\right\|_{L^2}\left\|\sqrt{\fr{\pr_tm_j}{m_j}}MD^i\right\|_{L^2}\\
&\le&c\eps\sum_{j=1,2,3}\left(\left\|\sqrt{\fr{\pr_tm_j}{m_j}}M\sqrt{-\Dl_L}\fr{B^i}{\eps}\right\|_{L^2}^2+\left\|\sqrt{\fr{\pr_tm_j}{m_j}}MD^i\right\|_{L^2}^2\right),
\eeq
and
\beq
\nn&&{-\fr12\sum_{j=1,2,3}\sum_{k,l}\int_\R \fr{\pr_tm}{m}\fr{\pr_tm_j}{m_j}M^2\hat{B}^i\bar{\hat{D}}^id\eta}\\
\nn&\le&\sum_{j=1,2,3}\left(\left\|\sqrt{\fr{\pr_tm}{m}\fr{\pr_tm_j}{m_j}}MB^i\right\|^2_{L^2}+\fr{1}{16}\left\|\sqrt{\fr{\pr_tm_j}{m_j}}MD^i\right\|^2_{L^2}\right).
\eeq
Moreover, in view of \eqref{wtw}--\eqref{ptp'}, we find that
\beq\label{role-m1}
\nn&&\fr{3c}{2}\sum_{k,l}\int_\R M^2\fr{\pr_tm}{m}\hat{B}^i\bar{\hat{D}}^id\eta-c\sum_{k,l}\int_\R M^2\fr{p'}{p}\hat{B}^i\bar{\hat{D}}^id\eta\\
\nn&&-\fr38\sum_{k,l}\int_\R M^2\left(\fr{\pr_tm}{m}\right)^2\hat{B}^i\bar{\hat{D}}^id\eta+\fr14\sum_{k,l}\int_\R \fr{\pr_tm}{m}M^2\fr{p'}{p}\hat{B}^i\bar{\hat{D}}^id\eta\\
\nn&&+\fr14\sum_{k,l}\int_\R M^2\pr_t\left(\fr{\pr_tm}{m}\right)\hat{B}^i\bar{\hat{D}}^id\eta\\
\nn&\le&\fr{3}{2}(c+\fr14)\sum_{k,l}\int_\R M^2\fr{\pr_tm}{m\sqrt{p}}\left(\sqrt{p}|\hat{B}^i|\right)|{\hat{D}}^i|d\eta\\
\nn&&+(c+\fr14)\int_\R M^2\fr{|p'|}{p^\fr32}\left(\sqrt{p}|\hat{B}^i|\right)|{\hat{D}}^i|d\eta+\fr{3}{2N}\left\|\sqrt{\fr{\pr_tm_1}{m_1}}MB^i\right\|_{L^2}\left\|\sqrt{\fr{\pr_tm_1}{m_1}}MD^i\right\|_{L^2}\\
\nn&\le&\fr{10(c+\fr14)\eps+3\eps}{2N}\left\|\sqrt{\fr{\pr_tm_1}{m_1}}M\sqrt{-\Dl_L}\fr{B^i}{\eps}\right\|_{L^2}\left\|\sqrt{\fr{\pr_tm_1}{m_1}}MD^i\right\|_{L^2}\\
&\le&\fr{(5c+3)\eps}{2N}\left(\left\|\sqrt{\fr{\pr_tm_1}{m_1}}M\sqrt{-\Dl_L}\fr{B^i}{\eps}\right\|_{L^2}^2+\left\|\sqrt{\fr{\pr_tm_1}{m_1}}MD^i\right\|_{L^2}^2\right).
\eeq
For $i=1$, by the definition of $m_1$, we have
\beq
\nn&&2c\sum_{k,l}\int_\R M^2\fr{k^2}{p}{\hat{B}}^1(\bar{\hat{W}}^2-\bar{\hat{B}}^1)d\eta-\fr12\sum_{k,l}\int_\R \fr{\pr_tm}{m}M^2\fr{k^2}{p}{\hat{B}}^1(\bar{\hat{W}}^2-\bar{\hat{B}}^1)d\eta\\
\nn&\le&\fr{2(c+\fr14)}{N}\left(\left\|\sqrt{\fr{\pr_tm_1}{m_1}}MW^2\right\|_{L^2}+\left\|\sqrt{\fr{\pr_tm_1}{m_1}}MB^1\right\|_{L^2}\right)\left\|\sqrt{\fr{\pr_tm_1}{m_1}}MB^1\right\|_{L^2}\\
&\le&\fr{c+\fr14}{N}\left\|\sqrt{\fr{\pr_tm_1}{m_1}}MW^2\right\|_{L^2}^2+\fr{3(c+\fr14)\eps^2}{N}\left\|\sqrt{\fr{\pr_tm_1}{m_1}}M\fr{B^1}{\eps}\right\|_{L^2}^2.
\eeq
For $i=3$, recalling the definition of $m_3$, there holds
\beq
\nn&&2c\sum_{k,l}\int_\R M^2\fr{kl}{p}{\hat{B}}^3(\bar{\hat{W}}^2-\bar{\hat{B}}^1)d\eta-\fr12\sum_{k,l}\int_\R \fr{\pr_tm}{m}M^2\fr{kl}{p}{\hat{B}}^3(\bar{\hat{W}}^2-\bar{\hat{B}}^1)d\eta\\
\nn&\le&\fr{2(c+\fr14)}{N}\left(\left\|\sqrt{\fr{\pr_tm_3}{m_3}}MW^2\right\|_{L^2}+\left\|\sqrt{\fr{\pr_tm_3}{m_3}}MB^1\right\|_{L^2}\right)\left\|\sqrt{\fr{\pr_tm_3}{m_3}}MB^3\right\|_{L^2}\\
&\le&\fr{c+\fr14}{N}\left(\left\|\sqrt{\fr{\pr_tm_3}{m_3}}MW^2\right\|_{L^2}^2+\eps^2\left\|\sqrt{\fr{\pr_tm_3}{m_3}}M\fr{B^1}{\eps}\right\|_{L^2}^2\right)+\fr{2(c+\fr14)\eps^2}{N}\left\|\sqrt{\fr{\pr_tm_3}{m_3}}M\fr{B^3}{\eps}\right\|_{L^2}^2.
\eeq
Finally, by virtue of the Cauchy-Schwarz inequality and \eqref{wtw},  we are led to
\be\label{cross-visco}
c(\lm+2\mu)\sum_{k,l}\int_\R M^2p\hat{B}^i\bar{\hat{D}}^id\eta
\le\fr{c}{4}\left\|M\sqrt{-\Dl_L}\fr{B^i}{\eps}\right\|_{L^2}^2+c(\lm+2\mu)^2\eps^2\|M\sqrt{-\Dl_L}D^i\|_{L^2}^2.
\ee
On the other hand, using the fact that
\be
\fr{\pr_tm}{m}\le\fr{|p'|}{p}\le\fr{2|k|}{\sqrt{p}}=\fr{2}{\sqrt{N}}\sqrt{\fr{\pr_tm_1}{m_1}},
\ee
we arrive at
\beq\label{cross-visco01}
\nn&&-\fr{\lm+2\mu}{4}\sum_{k,l}\int_\R \fr{\pr_tm}{m}M^2p\hat{B}^i\bar{\hat{D}}^id\eta\le \fr{\lm+2\mu}{2\sqrt{N}}\sum_{k,l}\int_\R \sqrt{\fr{\pr_tm_1}{m_1}}M^2p|\hat{B}^i||{\hat{D}}^i|d\eta\\
&\le&\fr{(\lm+2\mu)\eps^2}{2N}\left\|\sqrt{\fr{\pr_tm_1}{m_1}}M\sqrt{-\Dl_L}\fr{B^i}{\eps}\right\|_{L^2}^2+\fr{\lm+2\mu}{8}\|M\sqrt{-\Dl_L}D^i\|_{L^2}^2.
\eeq
It follows from \eqref{in-h'}--\eqref{cross-visco01} that
\beq\label{in-gh}
\nn&&\sum_{i=1,3}\fr{d}{dt}\Big(\la gB^i, gD^i\ra-\la hB^i, hD^i\ra\Big)\\
\nn&\le&-\fr{3c}{4}\sum_{i=1,3}\left\|M\sqrt{-\Dl_L}\fr{B^i}{\eps}\right\|^2_{L^2}+(\lm+2\mu)\Big(c(\lm+2\mu)\eps^2+\fr{1}{8}\Big)\sum_{i=1,3}\|M\sqrt{-\Dl_L}D^i\|_{L^2}^2\\
\nn&&+\fr14\sum_{i=1,3}\left\|\sqrt{\fr{\pr_tm}{m}}M\sqrt{-\Dl_L}\fr{B^i}{\eps}\right\|^2_{L^2}-\fr14\sum_{i=1,3}\left\|\sqrt{\fr{\pr_tm}{m}}MD^i\right\|^2_{L^2}+{c\sum_{i=1,3}\|MD^i\|^2_{L^2}}\\
\nn&&+c\eps\sum_{i=1,3}\left\|\sqrt{\fr{\pr_tm_2}{m_2}}M\sqrt{-\Dl_L}B^i\right\|_{L^2}^2+\Big(c\eps+\fr{1}{16}\Big)\left\|\sqrt{\fr{\pr_tm_2}{m_2}}MD^i\right\|_{L^2}^2\\
\nn&&+\Big(c\eps+\fr{1}{16}+\fr{\eps_2}{2N}\Big)\sum_{i=1,3}\sum_{j=1,3}\left(\left\|\sqrt{\fr{\pr_tm_j}{m_j}}M\sqrt{-\Dl_L}B^i\right\|_{L^2}^2+\left\|\sqrt{\fr{\pr_tm_j}{m_j}}MD^i\right\|_{L^2}^2\right)\\
&&+\fr{c+\fr14}{N}\sum_{j=1,3}\left\|\sqrt{\fr{\pr_tm_j}{m_j}}MW^2\right\|_{L^2}^2+\sum_{i=1,3}\sum_{j=1,2,3}\left\|\sqrt{\fr{\pr_tm}{m}\fr{\pr_tm_j}{m_j}}MB^i\right\|^2_{L^2},
\eeq
where 
\be\label{eps2}
\eps_2:=(5c+3)\eps+(6c+2)\eps^2+(\lm+2\mu)\eps^2.
\ee
Combining this with \eqref{Es}, we obtain
\beq\label{varE1}
&&\nn\fr12\fr{d}{dt}\mathcal{E}_s^1(t)+\fr{3c}{4}\sum_{i=1,3}\left\|M\sqrt{-\Dl_L}\fr{B^i}{\eps}\right\|^2_{L^2}\\
\nn&&+(\lm+2\mu)\Big(\fr78-c(\lm+2\mu)\eps^2\Big)\sum_{i=1,3}\|M\sqrt{-\Dl_L}D^i\|_{L^2}^2\\
\nn&&+\fr{\mu}{2}\left\|M\sqrt{-\Dl_L}W^2\right\|^2_{L^2}+\mu\left\|M_1\sqrt{-\Dl_L}W^3\right\|^2_{L^2}\\
\nn&&+\Big(\fr{15}{16}-c\eps\Big)\sum_{i=1,3}\left(\left\|\sqrt{\fr{\pr_tm_2}{m_2}}M\sqrt{-\Dl_L}\fr{B^i}{\eps}\right\|_{L^2}^2+\left\|\sqrt{\fr{\pr_tm_2}{m_2}}MD^i\right\|_{L^2}^2\right)\\
\nn&&+\left\|\sqrt{\fr{\pr_tm_2}{m_2}}MW^2\right\|^2_{L^2}+\left\|\sqrt{\fr{\pr_tm_2}{m_2}}M_1W^3\right\|^2_{L^2}+\fr34\sum_{i=1,3}\left\|{\fr{\pr_tm}{m}}M{B^i}\right\|^2_{L^2}\\
\nn&&+c_1\sum_{i=1,3}\sum_{j=1,3}\left(\left\|\sqrt{\fr{\pr_tm_j}{m_j}}M\sqrt{-\Dl_L}\fr{B^i}{\eps}\right\|_{L^2}^2+\left\|\sqrt{\fr{\pr_tm_j}{m_j}}MD^i\right\|_{L^2}^2\right)\\
\nn&&+c_2\sum_{j=1,3}\left(\left\|\sqrt{\fr{\pr_tm_j}{m_j}}MW^2\right\|^2_{L^2}+\left\|\sqrt{\fr{\pr_tm_j}{m_j}}M_1W^3\right\|^2_{L^2}\right)\\
&\le&{\fr{\mu^\fr13}{32}\left(\sum_{i=1,3}\left(\fr12\left\|M\sqrt{-\Dl_L}\fr{B^i}{\eps}\right\|^2_{L^2}+\left\|MD^i\right\|^2_{L^2}\right)+\left\|M_1W^3\right\|^2_{L^2}\right)}\\
\nn&&{+\fr{\mu\eps^2}{2}\left\|M\sqrt{-\Dl_L}\fr{B^1}{\eps}\right\|^2_{L^2}}+{c\sum_{i=1,3}\|MD^i\|^2_{L^2}},
\eeq
where
\[
\begin{split}
c_1=1-\fr{\eps_1}{N}-\left(c\eps+\fr{1}{16}+\fr{\eps_2}{2N}\right)
=\fr{15}{16}-c\eps-\fr{2\eps_1+\eps_2}{2N},
\end{split}
\]
and
\[
c_2=1-\fr{\eps_1}{N}-\fr{c+\fr14}{N}=1-\fr{1+4c+4\eps_1}{4N}.
\]
Let us take
\be\label{cN}
c=\fr{\mu^\fr13}{8}, \ \ \mathrm{and}\ \  N=9(1+\eps)^2.
\ee
Then recalling the definitions of $\eps_1$ and $\eps_2$ in \eqref{eps1} and \eqref{eps2}, respectively, and using \eqref{mu}, the following inequalities hold:
\be\label{small}
\begin{cases}
(\lm+2\mu)\left(\fr78-c(\lm+2\mu)\eps^2\right)\ge\fr12(\lm+2\mu)\ge\fr23\mu,\\[2mm]
\fr{15}{16}-c\eps\ge\fr12,\\[2mm]
\min\{c_1, c_2\}\ge\fr14,
\end{cases}
\ee
where we have used the fact that
\be
\lm+\fr{2}{3}\mu\ge0,\quad{\rm i.e.}\quad \mu\le\fr34(\lm+2\mu).
\ee
Then we infer from \eqref{varE1} that
\beq\label{varE2}
&&\nn\fr{d}{dt}\mathcal{E}^1_s(t)+\fr{3\mu^\fr13}{16}\sum_{i=1,3}\left\|M\sqrt{-\Dl_L}\fr{B^i}{\eps}\right\|^2_{L^2}\\
\nn&&+\mu\left(\sum_{i=1,3}\|M\sqrt{-\Dl_L}D^i\|_{L^2}^2+\left\|M\sqrt{-\Dl_L}W^2\right\|^2_{L^2}+\left\|M_1\sqrt{-\Dl_L}W^3\right\|^2_{L^2}\right)\\
\nn&&+\sum_{i=1,3}\left(\left\|\sqrt{\fr{\pr_tm_2}{m_2}}M\sqrt{-\Dl_L}\fr{B^i}{\eps}\right\|_{L^2}^2+\left\|\sqrt{\fr{\pr_tm_2}{m_2}}MD^i\right\|_{L^2}^2\right)\\
\nn&&+\left\|\sqrt{\fr{\pr_tm_2}{m_2}}MW^2\right\|^2_{L^2}+\left\|\sqrt{\fr{\pr_tm_2}{m_2}}M_1W^3\right\|^2_{L^2}\\
\nn&&+\fr12\sum_{i=1,3}\sum_{j=1,3}\left(\left\|\sqrt{\fr{\pr_tm_j}{m_j}}M\sqrt{-\Dl_L}\fr{B^i}{\eps}\right\|_{L^2}^2+\left\|\sqrt{\fr{\pr_tm_j}{m_j}}MD^i\right\|_{L^2}^2\right)\\
\nn&&+\fr12\sum_{j=1,3}\left(\left\|\sqrt{\fr{\pr_tm_j}{m_j}}MW^2\right\|^2_{L^2}+\left\|\sqrt{\fr{\pr_tm_j}{m_j}}M_1W^3\right\|^2_{L^2}\right)\\
&\le&{\fr{\mu^\fr13}{16}\left(\sum_{i=1,3}\left(\fr12\left\|M\sqrt{-\Dl_L}\fr{B^i}{\eps}\right\|^2_{L^2}+\left\|MD^i\right\|^2_{L^2}\right)+\left\|M_1W^3\right\|^2_{L^2}\right)}\\
\nn&&{+{\mu\eps^2}\left\|M\sqrt{-\Dl_L}\fr{B^1}{\eps}\right\|^2_{L^2}}+{\fr{\mu^\fr13}{4}\sum_{i=1,3}\|MD^i\|^2_{L^2}}.
\eeq
Thanks to \eqref{claim}, one deduces that
\beq\label{en-dis}
\nn&&\fr{3\mu}{4}\left(\sum_{i=1,3}\|M\sqrt{-\Dl_L}D^i\|_{L^2}^2+\left\|M\sqrt{-\Dl_L}W^2\right\|^2_{L^2}+\left\|M_1\sqrt{-\Dl_L}W^3\right\|^2_{L^2}\right)\\
\nn&&+\fr34\left(\sum_{i=1,3}\left\|\sqrt{\fr{\pr_tm_2}{m_2}}MD^i\right\|_{L^2}^2+\left\|\sqrt{\fr{\pr_tm_2}{m_2}}MW^2\right\|^2_{L^2}+\left\|\sqrt{\fr{\pr_tm_2}{m_2}}M_1W^3\right\|^2_{L^2}\right)\\
&\ge&\fr{3\mu^\fr13}{8}\left(\sum_{i=1,3}\|MD^i\|_{L^2}^2+\left\|MW^2\right\|^2_{L^2}+\left\|M_1W^3\right\|^2_{L^2}\right).
\eeq
On the other hand, thanks to \eqref{wtwp}, we find that
\be\label{dis-erro}
\fr14\sum_{i=1,3}\left\|\sqrt{\fr{\pr_tm_1}{m_1}}M\sqrt{-\Dl_L}\fr{B^i}{\eps}\right\|_{L^2}^2\ge\fr{N}{8\eps^2}\sum_{i=1,3}\left\|\sqrt{\fr{\pr_tm}{m}}M{B^i}\right\|_{L^2}^2\ge\fr{\mu^\fr13}{16}\sum_{i=1,3}\left\|\sqrt{\fr{\pr_tm}{m}}M{B^i}\right\|_{L^2}^2.
\ee
Noting that \eqref{mu} implies that $\mu\eps^2\le\fr{\mu^\fr13}{16}$, then the right hand side of \eqref{varE2} can be bounded as follows
\beq\label{rhs}
\nn&&{\fr{\mu^\fr13}{16}\left(\sum_{i=1,3}\left(\fr12\left\|M\sqrt{-\Dl_L}\fr{B^i}{\eps}\right\|^2_{L^2}+\left\|MD^i\right\|^2_{L^2}\right)+\left\|M_1W^3\right\|^2_{L^2}\right)}\\
\nn&&{+{\mu\eps^2}\left\|M\sqrt{-\Dl_L}\fr{B^1}{\eps}\right\|^2_{L^2}}+{\fr{\mu^\fr13}{4}\sum_{i=1,3}\|MD^i\|^2_{L^2}}\\
&\le&\fr{3\mu^\fr13}{32}\sum_{i=1,3}\left\|M\sqrt{-\Dl_L}\fr{B^i}{\eps}\right\|^2_{L^2}+\fr{5\mu^\fr13}{16}\left(\sum_{i=1,3}\|MD^i\|_{L^2}^2+\left\|M_1W^3\right\|^2_{L^2}\right).
\eeq
Now substituting \eqref{en-dis}, \eqref{dis-erro} and \eqref{rhs} into \eqref{varE2}, we conclude that
\beq\label{varE3}
&&\nn\fr{d}{dt}\mathcal{E}_s^1(t)+\fr{\mu^\fr13}{16}E^1_s(t)\\
\nn&&+\fr{\mu}{4}\left(\sum_{i=1,3}\|M\sqrt{-\Dl_L}D^i\|_{L^2}^2+\left\|M\sqrt{-\Dl_L}W^2\right\|^2_{L^2}+\left\|M_1\sqrt{-\Dl_L}W^3\right\|^2_{L^2}\right)\\
\nn&&+\fr14\sum_{i=1,3}\sum_{j=1,2,3}\left(\left\|\sqrt{\fr{\pr_tm_j}{m_j}}M\sqrt{-\Dl_L}\fr{B^i}{\eps}\right\|_{L^2}^2+\left\|\sqrt{\fr{\pr_tm_j}{m_j}}MD^i\right\|_{L^2}^2\right)\\
&&+\fr14\sum_{j=1,2,3}\left(\left\|\sqrt{\fr{\pr_tm_j}{m_j}}MW^2\right\|^2_{L^2}+\left\|\sqrt{\fr{\pr_tm_j}{m_j}}M_1W^3\right\|^2_{L^2}\right)\le0.
\eeq
Recalling the choice of $c$,  using \eqref{wtw} and \eqref{mu}, we see that
\beqno
&&2\left|\sum_{i=1,3}\left(\la gB^i, gD^i\ra-\la hB^i, hD^i\ra\right)\right|\\
&\le&\fr14\sum_{i=1,3}\left(\left\|\sqrt{\fr{\pr_tm}{m}}MB^i\right\|^2_{L^2}+\|MD^i\|^2_{L^2}\right)+\fr{\mu^\fr13\eps}{8}\left(\left\|M\sqrt{-\Dl_L}\fr{B^i}{\eps}\right\|^2_{L^2}+\|MD^i\|^2_{L^2}\right)\\
&\le&\fr38\sum_{i=1,3}\left(\left\|M\sqrt{-\Dl_L}\fr{B^i}{\eps}\right\|^2_{L^2}+\left\|\sqrt{\fr{\pr_tm}{m}}MB^i\right\|^2_{L^2}+\|MD^i\|_{L^2}^2\right).
\eeqno
Accordingly,
\be\label{app1}
\fr58E^1_s(t)\le\mathcal{E}_s^1(t)\le\fr{11}{8}E^1_s(t).
\ee
Combining \eqref{app1} with \eqref{varE3}, the inequality \eqref{enhan13} in Proposition \ref{prop-endis} follows immediately.

\subsection{Estimates of $(B^2, D^2, W^1)$}\label{subsec-4.2}
Define
\be
E_s^2(t):=\left\|M\sqrt{-\Dl_L}\fr{B^2}{\eps}\right\|^2_{L^2}+\left\|\sqrt{\fr{\pr_tm}{m}}M{B^2}\right\|^2_{L^2}+\|MD^2\|^2_{L^2}+\|M_1W^1\|^2_{L^2},
\ee
and
\be
\mathcal{E}_s^2(t):=E^2_s(t)+2\left(\la gB^2, gD^2\ra-\la hB^2, hD^2\ra\right),
\ee
where  the parameter $c$ in the weight $h$ is  chosen as before.
Similar to \eqref{e1}, it is not difficult to verify that
\beq\label{e1'}
\nn&&\fr{1}{2}\fr{d}{dt}E^2_s(t)+(\lm+2\mu)\left\|M\sqrt{-\Dl_L}D^2\right\|^2_{L^2}+\mu\left\|M_1\sqrt{-\Dl_L}W^1\right\|^2_{L^2}\\
\nn&&+\fr34\left\|\fr{\pr_tm}{m}MB^2\right\|^2_{L^2}+\sum_{j=1,2,3}\left\|\sqrt{\fr{\pr_tm}{m}\fr{\pr_tm_j}{m_j}}MB^2\right\|^2_{L^2}\\
\nn&&+\fr34\left(\left\|\sqrt{\fr{\pr_tm}{m}}M\sqrt{-\Dl_L}\fr{B^2}{\eps}\right\|^2_{L^2}+\left\|\sqrt{\fr{\pr_tm}{m}}MD^2\right\|^2_{L^2}\right)+\left\|\sqrt{\fr{\pr_tm}{m}}M_1W^1\right\|^2_{L^2}\\
\nn&&+\sum_{j=1,2,3}\left(\left\|\sqrt{\fr{\pr_tm_j}{m_j}}M\sqrt{-\Dl_L}\fr{B^2}{\eps}\right\|^2_{L^2}+\left\|\sqrt{\fr{\pr_tm_j}{m_j}}MD^2\right\|^2_{L^2}+\left\|\sqrt{\fr{\pr_tm_j}{m_j}}M_1W^1\right\|^2_{L^2}\right)\\
\nn&=&\fr12\sum_{k, l}\int_\mathbb{R}M^2p'\left|\fr{{\hat{B}^2}}{\eps}\right|^2d\eta+\sum_{k,l}\int_\R M^2\fr{p'}{p}|\hat{D}^2|^2d\eta+\sum_{k, l}\int_\R M_1^2\fr{p'}{p}|\hat{W}^1|^2d\eta\\
\nn&&+\sum_{k, l}\int_\R M_1^2\fr{p'}{p}\hat{B}^2\bar{\hat{W}}^1d\eta-\mu\left\la M_1\sqrt{-\Dl_L}B^2, M_1\sqrt{-\Dl_L}W^1\right\ra\\
\nn&&+\fr12\sum_{k,l}\int_{\mathbb{R}}\pr_t\left(\fr{\pr_tm}{m}\right)M^2|\hat{B}^2|^2d\eta-\sum_{k,l}\int_{\mathbb{R}}\fr{\pr_tm}{m}M^2\hat{B}^2\bar{\hat{D}}^2d\eta+\sum_{k,l}\int_{\mathbb{R}}\fr{\pr_tm}{m}M^2\hat{F}\bar{\hat{B}}^2d\eta\\
&&+{\fr{1}{\eps^2}\left\la M\sqrt{-\Dl_L}F,  M\sqrt{-\Dl_L}B^2\right\ra}+\left\la MG, MD^2\right\ra+\left\la M_1H, M_1W^1\right\ra.
\eeq
Using the facts  $\fr{|p'|}{p^{\fr32}}\le\fr{2|k|}{p}\le\fr{2}{N}\fr{\pr_tm_1}{m_1}$, and $M_1\le M$, we arrive at
\beq
\nn\sum_{k, l}\int_\R M_1^2\fr{p'}{p}\hat{B}^2\bar{\hat{W}}^1d\eta&\le&\sum_{k, l}\int_\R M_1^2\fr{p'}{p^{\fr32}}\left(\sqrt{p}\left|\hat{B}^2\right|\right)\left|{\hat{W}}^1\right|d\eta\\
\nn&\le&\fr{2}{N}\left\|\sqrt{\fr{\pr_tm_1}{m_1}}M_1\sqrt{-\Dl_L}B^2\right\|_{L^2}\left\|\sqrt{\fr{\pr_tm_1}{m_1}}M_1W^1\right\|_{L^2}\\
&\le&\fr{\eps}{N}\left(\left\|\sqrt{\fr{\pr_tm_1}{m_1}}M\sqrt{-\Dl_L}\fr{B^2}{\eps}\right\|_{L^2}^2+\left\|\sqrt{\fr{\pr_tm_1}{m_1}}M_1W^1\right\|_{L^2}^2\right).
\eeq
Similar to \eqref{Bi1} and \eqref{Bi2}, we have
\be
\fr12\sum_{k,l}\int_{\mathbb{R}}\pr_t\left(\fr{\pr_tm}{m}\right)M^2|\hat{B}^2|^2d\eta\le\fr{3\eps^2}{N}\left\|\sqrt{\fr{\pr_tm_1}{m_1}}M\sqrt{-\Dl_L}\fr{B^2}{\eps}\right\|^2_{L^2},
\ee
and
\beq
\nn&&-\sum_{k,l}\int_{\mathbb{R}}\fr{\pr_tm}{m}M^2\hat{B}^2\bar{\hat{D}}^2d\eta+\sum_{k,l}\int_{\mathbb{R}}\fr{\pr_tm}{m}M^2\hat{F}\bar{\hat{B}}^2d\eta\\
\nn&\le&\fr{2}{N}\left\|\sqrt{\fr{\pr_tm_1}{m_1}}M\sqrt{-\Dl_L}B^2\right\|_{L^2}\left(\left\|\sqrt{\fr{\pr_tm_1}{m_1}}MD^2\right\|_{L^2}+\left\|\sqrt{\fr{\pr_tm_1}{m_1}}M\sqrt{-\Dl_L}F\right\|_{L^2}\right)\\
&\le&\fr{\eps(1+\eps)}{N}\left\|\sqrt{\fr{\pr_tm_1}{m_1}}M\sqrt{-\Dl_L}\fr{B^2}{\eps}\right\|_{L^2}^2+\fr{\eps}{N}\left\|\sqrt{\fr{\pr_tm_1}{m_1}}MD^2\right\|^2_{L^2}+\fr{\eps^2}{N}\left\|\sqrt{\fr{\pr_tm_1}{m_1}}M\sqrt{-\Dl_L}\fr{F}{\eps}\right\|_{L^2}^2.
\eeq
The Cauchy-Schwarz inequality implies that
\be\label{CS}
\begin{split}
&{\fr{1}{\eps^2}\left\la M\sqrt{-\Dl_L}F,  M\sqrt{-\Dl_L}B^2\right\ra}+\left\la MG, MD^2\right\ra+{\left\la M_1H, M_1W^1\right\ra}\\
\le&{\fr{\mu^\fr13}{256}\left(\left\|M\sqrt{-\Dl_L}\fr{B^2}{\eps}\right\|^2_{L^2}+\|MD^2\|^2_{L^2}+\|M_1W^1\|^2_{L^2}\right)}\\
&+\fr{64}{\mu^\fr13}\left(\left\|M\sqrt{-\Dl_L}\fr{F}{\eps}\right\|^2_{L^2}+\|MG\|^2_{L^2}+{\|M_1H\|^2_{L^2}}\right).
\end{split}
\ee
Then arguing as \eqref{role-m}, one deduces from  \eqref{e1'}--\eqref{CS} that
\beq\label{e1-2}
\nn&&\fr{1}{2}\fr{d}{dt}E^2_s(t)+(\lm+2\mu)\left\|M\sqrt{-\Dl_L}D^2\right\|^2_{L^2}+\fr{\mu}{2}\left\|M_1\sqrt{-\Dl_L}W^1\right\|^2_{L^2}\\
\nn&&+\fr34\left\|\fr{\pr_tm}{m}MB^2\right\|^2_{L^2}+\sum_{j=1,2,3}\left\|\sqrt{\fr{\pr_tm}{m}\fr{\pr_tm_j}{m_j}}MB^2\right\|^2_{L^2}\\
\nn&&+\left(\left\|\sqrt{\fr{\pr_tm_2}{m_2}}M\sqrt{-\Dl_L}\fr{B^2}{\eps}\right\|^2_{L^2}+\left\|\sqrt{\fr{\pr_tm_2}{m_2}}MD^2\right\|^2_{L^2}+\left\|\sqrt{\fr{\pr_tm_2}{m_2}}M_1W^1\right\|^2_{L^2}\right)\\
\nn&&+\left(1-\fr{\eps_3}{N}\right)\sum_{j=1,3}\left(\left\|\sqrt{\fr{\pr_tm_j}{m_j}}M\sqrt{-\Dl_L}\fr{B^2}{\eps}\right\|^2_{L^2}+\left\|\sqrt{\fr{\pr_tm_j}{m_j}}MD^2\right\|^2_{L^2}+\left\|\sqrt{\fr{\pr_tm_j}{m_j}}M_1W^1\right\|^2_{L^2}\right)\\
\nn&\le&\fr14\left(\left\|\sqrt{\fr{\pr_tm}{m}}MD^2\right\|^2_{L^2}-\left\|\sqrt{\fr{\pr_tm}{m}}M\sqrt{-\Dl_L}\fr{B^2}{\eps}\right\|^2_{L^2}\right)\\
\nn&&+{\left(\fr{5\mu^\fr13}{256}+\fr{\mu\eps^2}{2}\right)\left\|M\sqrt{-\Dl_L}\fr{B^2}{\eps}\right\|^2_{L^2}+\fr{9\mu^\fr13}{256}\left(\|MD^2\|^2_{L^2}+\|M_1W^1\|^2_{L^2}\right)}\\
&&+\fr{64}{\mu^\fr13}\left(\left\|M\sqrt{-\Dl_L}\fr{F}{\eps}\right\|^2_{L^2}+\|MG\|^2_{L^2}+{\|M_1H\|^2_{L^2}}\right)+\fr{\eps^2}{N}\left\|\sqrt{\fr{\pr_tm_1}{m_1}}M\sqrt{-\Dl_L}\fr{F}{\eps}\right\|_{L^2}^2,
\eeq
where 
\be
\eps_3:=4\eps^2+2\eps.
\ee
Next, direct calculations show that the evolutions of $\la hB^2, hD^2\ra$ and $\la gB^2, gD^2$ are given as follows
\beq\label{in-h2}
\nn\fr{d}{dt}\la hB^2, hD^2\ra&=&c\left\|M\sqrt{-\Dl_L}\fr{B^2}{\eps}\right\|^2_{L^2}-c\|MD^2\|^2_{L^2}-c(\lm+2\mu)\sum_{k,l}\int_\R pM^2\hat{B}^2\bar{\hat{D}}^2d\eta\\
\nn&&+c\sum_{k,l}\int_\R M^2\fr{p'}{p}\hat{B}^2\bar{\hat{D}}^2d\eta-\fr{3c}{2}\sum_{k,l}\int_\R \fr{\pr_tm}{m}M^2\hat{B}^2\bar{\hat{D}}^2d\eta\\
&&-2c\sum_{j=1,2,3}\sum_{k,l}\int_\R M^2\fr{\pr_tm_j}{m_j}\hat{B}^2\bar{\hat{D}}^2d\eta+c\left\la MF,  MD^2\right\ra+c\left\la MB^2, MG\right\ra,
\eeq
and
\beq\label{in-g2}
\nn&&\fr{d}{dt}\la gB^2, gD^2\ra\\
\nn&=&\fr14\left(\left\|\sqrt{\fr{\pr_tm}{m}}M\sqrt{-\Dl_L}\fr{B^2}{\eps}\right\|^2_{L^2}-\left\|\sqrt{\fr{\pr_tm}{m}}MD^2\right\|^2_{L^2}\right)-\fr{\lm+2\mu}{4}\sum_{k,l}\int_\R p\fr{\pr_tm}{m}M^2\hat{B}^2\bar{\hat{D}}^2d\eta\\
\nn&&+\fr14\sum_{k,l}\int_\R \fr{\pr_tm}{m}M^2\fr{p'}{p}\hat{B}^2\bar{\hat{D}}^2d\eta-\fr38\sum_{k,l}\int_\R \left(\fr{\pr_tm}{m}\right)^2M^2\hat{B}^2\bar{\hat{D}}^2d\eta\\
\nn&&+\fr14\sum_{k,l}\int_\R\pr_t\left(\fr{\pr_tm}{m}\right)M^2\hat{B}^2\bar{\hat{D}}^2d\eta-\fr12\sum_{j=1,2,3}\sum_{k,l}\int_\R \fr{\pr_tm}{m}\fr{\pr_tm_j}{m_j}M^2\hat{B}^2\bar{\hat{D}}^2d\eta\\
&&+\fr14\left\la \sqrt{\fr{\pr_tm}{m}}MF,  \sqrt{\fr{\pr_tm}{m}}MD^2\right\ra+\fr14\left\la \sqrt{\fr{\pr_tm}{m}}MB^2, \sqrt{\fr{\pr_tm}{m}}MG\right\ra.
\eeq
Recalling that $c=\fr{\mu^\fr13}{8}$, using the Cauchy-Schwarz inequality,  we have
\beq\label{sor1}
\nn&&-c\left\la MF,  MD^2\right\ra-c\left\la MB^2, MG\right\ra\\
&\le&\fr{\mu^\fr13}{256}\left(\left\|M\sqrt{-\Dl_L}\fr{B^2}{\eps}\right\|^2_{L^2}+\|MD^2\|^2_{L^2}\right)+\fr{(\mu^\fr13\eps)^2}{\mu^\fr13}\left(\left\|M\fr{F}{\eps}\right\|^2_{L^2}+\|MG\|^2_{L^2}\right).
\eeq
Using \eqref{wtwp} again, we obtain
\beq\label{sor-01}
\nn&&\fr14\left\la \sqrt{\fr{\pr_tm}{m}}MF,  \sqrt{\fr{\pr_tm}{m}}MD^2\right\ra+\fr14\left\la \sqrt{\fr{\pr_tm}{m}}MB^2, \sqrt{\fr{\pr_tm}{m}}MG\right\ra\\
\nn&\le&\fr{\eps}{4N}\left(\left\|\sqrt{\fr{\pr_tm_1}{m_1}}M\sqrt{-\Dl_L}\fr{B^2}{\eps}\right\|^2_{L^2}+\left\|\sqrt{\fr{\pr_tm_1}{m_1}}MD^2\right\|^2_{L^2}\right.\\
&&\left.+\left\|\sqrt{\fr{\pr_tm_1}{m_1}}M\sqrt{-\Dl_L}\fr{F}{\eps}\right\|^2_{L^2}+\left\|\sqrt{\fr{\pr_tm_1}{m_1}}MG\right\|^2_{L^2}\right).
\eeq
Now arguing as \eqref{role-mj},  \eqref{role-m1} and \eqref{cross-visco},   from \eqref{in-h2}--\eqref{sor-01}, we find  that
\beq\label{in-gh2}
\nn&&\fr{d}{dt}\Big(\la gB^2, gD^2\ra-\la hB^2, hD^2\ra\Big)\\
\nn&\le&\fr{1}{4}\left(\left\|\sqrt{\fr{\pr_tm}{m}}M\sqrt{-\Dl_L}\fr{B^2}{\eps}\right\|^2_{L^2}-\left\|\sqrt{\fr{\pr_tm}{m}}MD^2\right\|^2_{L^2}\right)\\
\nn&&-\fr{23\mu^\fr13}{256}\left\|M\sqrt{-\Dl_L}\fr{{B}^2}{\eps}\right\|^2_{L^2}+(\lm+2\mu)\Big(c(\lm+2\mu)\eps^2+\fr{1}{8}\Big)\|M\sqrt{-\Dl_L}D^2\|_{L^2}^2\\
\nn&&+c\eps\left\|\sqrt{\fr{\pr_tm_2}{m_2}}M\sqrt{-\Dl_L}\fr{B^2}{\eps}\right\|_{L^2}^2+\Big(c\eps+\fr{1}{16}\Big)\left\|\sqrt{\fr{\pr_tm_2}{m_2}}MD^2\right\|_{L^2}^2+\sum_{j=1,2,3}\left\|\sqrt{\fr{\pr_tm}{m}\fr{\pr_tm_j}{m_j}}MB^2\right\|^2_{L^2}\\
\nn&&+\Big(c\eps+\fr{1}{16}+\fr{\eps_4}{4N}\Big)\sum_{j=1,3}\left(\left\|\sqrt{\fr{\pr_tm_j}{m_j}}M\sqrt{-\Dl_L}\fr{B^2}{\eps}\right\|_{L^2}^2+\left\|\sqrt{\fr{\pr_tm_j}{m_j}}MD^2\right\|_{L^2}^2\right)\\
\nn&&+{\fr{33\mu^\fr13}{256}\|MD^2\|^2_{L^2}}+\fr{(\mu^\fr13\eps)^2}{\mu^\fr13}\left(\left\|M\sqrt{-\Dl_L}\fr{F}{\eps}\right\|^2_{L^2}+\|MG\|^2_{L^2}\right)\\
&&+\fr{\eps}{4N}\left(\left\|\sqrt{\fr{\pr_tm_1}{m_1}}M\sqrt{-\Dl_L}\fr{F}{\eps}\right\|^2_{L^2}+\left\|\sqrt{\fr{\pr_tm_1}{m_1}}MG\right\|^2_{L^2}\right),
\eeq
where
\be\label{eps4}
\eps_4:=2(5c+3)\eps+2(\lm+2\mu)\eps^2+\eps.
\ee
Noting that 
\[
\fr{15}{16}-c\eps-\fr{4\eps_3+\eps_4}{4N}\ge c_1,
\]
then  \eqref{small} holds. Thus, we infer from \eqref{e1-2} and \eqref{in-gh2} that
\beq\label{varE2-2}
&&\nn\fr{d}{dt}\mathcal{E}^2_s(t)+\fr{9\mu^\fr13}{64}\left\|M\sqrt{-\Dl_L}\fr{B^2}{\eps}\right\|^2_{L^2}+\mu\left(\|M\sqrt{-\Dl_L}D^2\|_{L^2}^2+\left\|M_1\sqrt{-\Dl_L}W^1\right\|^2_{L^2}\right)\\
\nn&&+\left\|\sqrt{\fr{\pr_tm_2}{m_2}}M\sqrt{-\Dl_L}\fr{B^2}{\eps}\right\|_{L^2}^2+\left\|\sqrt{\fr{\pr_tm_2}{m_2}}MD^2\right\|_{L^2}^2+\left\|\sqrt{\fr{\pr_tm_2}{m_2}}M_1W^1\right\|^2_{L^2}\\
\nn&&+\fr12\sum_{j=1,3}\left(\left\|\sqrt{\fr{\pr_tm_j}{m_j}}M\sqrt{-\Dl_L}\fr{B^2}{\eps}\right\|_{L^2}^2+\left\|\sqrt{\fr{\pr_tm_j}{m_j}}MD^2\right\|_{L^2}^2+\left\|\sqrt{\fr{\pr_tm_j}{m_j}}M_1W^1\right\|^2_{L^2}\right)\\
\nn&\le&{{\mu\eps^2}\left\|M\sqrt{-\Dl_L}\fr{B^2}{\eps}\right\|^2_{L^2}+\fr{21\mu^\fr13}{64}\left(\|MD^2\|^2_{L^2}+\|M_1W^1\|^2_{L^2}\right)}\\
\nn&&+\fr{130}{\mu^\fr13}\left(\left\|M\sqrt{-\Dl_L}\fr{F}{\eps}\right\|^2_{L^2}+\|MG\|^2_{L^2}+\|M_1H\|^2_{L^2}\right)\\
&&+\fr{\eps+4\eps^2}{2N}\left(\left\|\sqrt{\fr{\pr_tm_1}{m_1}}M\sqrt{-\Dl_L}\fr{F}{\eps}\right\|^2_{L^2}+\left\|\sqrt{\fr{\pr_tm_1}{m_1}}MG\right\|^2_{L^2}\right).
\eeq
Then it follows from \eqref{claim}, \eqref{mu} and \eqref{varE2-2} that 
\beq\label{varE2-3}
&&\nn\fr{d}{dt}\mathcal{E}^2_s(t)+\fr{\mu^\fr13}{32}E_s^2(t)+\fr{\mu}{4}\left(\|M\sqrt{-\Dl_L}D^2\|_{L^2}^2+\left\|M_1\sqrt{-\Dl_L}W^1\right\|^2_{L^2}\right)\\
\nn&&+\fr14\sum_{j=1,2,3}\left(\left\|\sqrt{\fr{\pr_tm_j}{m_j}}M\sqrt{-\Dl_L}\fr{B^2}{\eps}\right\|_{L^2}^2+\left\|\sqrt{\fr{\pr_tm_j}{m_j}}MD^2\right\|_{L^2}^2+\left\|\sqrt{\fr{\pr_tm_j}{m_j}}M_1W^1\right\|^2_{L^2}\right)\\
\nn&\le&\fr{130}{\mu^\fr13}\left(\left\|M\sqrt{-\Dl_L}\fr{F}{\eps}\right\|^2_{L^2}+\|MG\|^2_{L^2}+\|M_1H\|^2_{L^2}\right)\\
&&+\fr{\eps+4\eps^2}{2N}\left(\left\|\sqrt{\fr{\pr_tm_1}{m_1}}M\sqrt{-\Dl_L}\fr{F}{\eps}\right\|^2_{L^2}+\left\|\sqrt{\fr{\pr_tm_1}{m_1}}MG\right\|^2_{L^2}\right),
\eeq
where we have used \eqref{dis-erro} with $i=1,3$ replaced by $i=2$. Obviously, by the definition of the multiplier $m_1$, for $k\neq0$, there holds $\fr{1}{p}\le\fr{1}{N}\fr{\pr_tm_1}{m_1}$. Thus, 
\beno
\|MB^1\|^2_{L^2}=\sum_{k,l}\int_{\mathbb{R}}\fr{1}{p}M^2|\sqrt{p}\hat{B}^1|^2d\eta\le\fr{\eps^2}{N}\left\|\sqrt{\fr{\pr_tm_1}{m_1}}M\sqrt{-\Dl_L}\fr{B^1}{\eps}\right\|_{L^2}^2.
\eeno
Accordingly, recalling the definitions of $F, G$ and $H$, it is easy to verify that 

\beq
\nn&&\left\|M\sqrt{-\Dl_L}\fr{F}{\eps}\right\|^2_{L^2}+\|MG\|^2_{L^2}+\|M_1H\|^2_{L^2}\\
\nn&\le&\left\|M\sqrt{-\Dl_L}\fr{B^1}{\eps}\right\|^2_{L^2}+35\left(\|MD^1\|_{L^2}^2+\|MW^2\|_{L^2}^2+\|MB^1\|_{L^2}^2\right)\\
&\le&35\left(\left\|M\sqrt{-\Dl_L}\fr{B^1}{\eps}\right\|^2_{L^2}+\|MD^1\|_{L^2}^2+\|MW^2\|_{L^2}^2+\fr{\eps^2}{N}\left\|\sqrt{\fr{\pr_tm_1}{m_1}}M\sqrt{-\Dl_L}\fr{B^1}{\eps}\right\|_{L^2}^2\right).
\eeq
Substituting this into \eqref{varE2-3}, and recalling that $N$ is chosen to be $9(1+\eps)^2$ in \eqref{cN}, we are led to
\beq\label{varE2-4}
&&\nn\fr{d}{dt}\mathcal{E}^2_s(t)+\fr{\mu^\fr13}{32}E_s^2(t)+\fr{\mu}{4}\left(\|M\sqrt{-\Dl_L}D^2\|_{L^2}^2+\left\|M_1\sqrt{-\Dl_L}W^1\right\|^2_{L^2}\right)\\
\nn&&+\fr14\sum_{j=1,2,3}\left(\left\|\sqrt{\fr{\pr_tm_j}{m_j}}M\sqrt{-\Dl_L}\fr{B^2}{\eps}\right\|_{L^2}^2+\left\|\sqrt{\fr{\pr_tm_j}{m_j}}MD^2\right\|_{L^2}^2+\left\|\sqrt{\fr{\pr_tm_j}{m_j}}M_1W^1\right\|^2_{L^2}\right)\\
\nn&\le&\fr{4550}{\mu^\fr13}\left(\left\|M\sqrt{-\Dl_L}\fr{B^1}{\eps}\right\|^2_{L^2}+\|MD^1\|_{L^2}+\|MW^2\|_{L^2}^2+\left\|\sqrt{\fr{\pr_tm_1}{m_1}}M\sqrt{-\Dl_L}\fr{B^1}{\eps}\right\|_{L^2}^2\right)\\
&&+\left((1+\eps^2)\left\|\sqrt{\fr{\pr_tm_1}{m_1}}M\sqrt{-\Dl_L}\fr{B^1}{\eps}\right\|^2_{L^2}+\left\|\sqrt{\fr{\pr_tm_1}{m_1}}MD^1\right\|^2_{L^2}+\left\|\sqrt{\fr{\pr_tm_1}{m_1}}MW^2\right\|^2_{L^2}\right).
\eeq
Multiplying \eqref{varE2-4} by $c_0\mu^{\fr23}$ with positive constant $c_0$ so small that
\be\label{c_0}
c_0\le\fr{1}{32\times4550},
\ee  
then adding the resulting inequality to \eqref{varE3}, and  using \eqref{mu} again, one deduces that
\beq\label{varE2-5}
&&\nn\fr{d}{dt}\left(\mathcal{E}_s^1(t)+c_0\mu^{\fr23}\mathcal{E}^2_s(t)\right)+\fr{\mu^\fr13}{32}\left(E^1_s(t)+c_0\mu^{\fr23}E_s^2(t)\right)\\
\nn&&+\fr{\mu}{4}\left(\sum_{i=1,3}\|M\sqrt{-\Dl_L}D^i\|_{L^2}^2+\left\|M\sqrt{-\Dl_L}W^2\right\|^2_{L^2}+\left\|M_1\sqrt{-\Dl_L}W^3\right\|^2_{L^2}\right)\\
\nn&&+\fr{c_0\mu^{\fr{5}{3}}}{4}\left(\|M\sqrt{-\Dl_L}D^2\|_{L^2}^2+\left\|M_1\sqrt{-\Dl_L}W^1\right\|^2_{L^2}\right)\\
\nn&&+\fr18\sum_{i=1,3}\sum_{j=1,2,3}\left(\left\|\sqrt{\fr{\pr_tm_j}{m_j}}M\sqrt{-\Dl_L}\fr{B^i}{\eps}\right\|_{L^2}^2+\left\|\sqrt{\fr{\pr_tm_j}{m_j}}MD^i\right\|_{L^2}^2\right)\\
\nn&&+\fr18\sum_{j=1,2,3}\left(\left\|\sqrt{\fr{\pr_tm_j}{m_j}}MW^2\right\|^2_{L^2}+\left\|\sqrt{\fr{\pr_tm_j}{m_j}}M_1W^3\right\|^2_{L^2}\right)\\
\nn&&+\fr{c_0\mu^{\fr{2}{3}}}{4}\sum_{j=1,2,3}\left(\left\|\sqrt{\fr{\pr_tm_j}{m_j}}M\sqrt{-\Dl_L}\fr{B^2}{\eps}\right\|_{L^2}^2+\left\|\sqrt{\fr{\pr_tm_j}{m_j}}MD^2\right\|_{L^2}^2+\left\|\sqrt{\fr{\pr_tm_j}{m_j}}M_1W^1\right\|^2_{L^2}\right)\le0.
\eeq
Obviously, the relation \eqref{app1} holds for ${E}_s^2(t)$ and $\mathcal{E}^2_s(t)$ as well, and thus one easily obtain \eqref{enhan2}. We complete the proof Proposition \ref{prop-endis}.  

\section{Proof of the main results}
In this section, we complete the proof of  Theorems \ref{coro1}--\ref{thm-0}  one by one. As pointed out in Remark \ref{rem0}, we shall prove more than we need in Theorem \ref{coro1} below. To begin with, recalling the change of coordinates \eqref{ch-co}, and using \eqref{m-1}, \eqref{m-2} and \eqref{enhan13}, we arrive at
\beq
\nn\|\pr_xb\|_{L^2}&=&\left\|B^1_\neq\right\|_{L^2}\\
\nn&\le&\left\|m^{-\fr12}\sqrt{-\Dl_L}B^1_\neq\right\|_{L^2}\\
\nn&\le&C\mu^{-\fr16}\left\|m^{-\fr34}\sqrt{-\Dl_L}B^1_\neq\right\|_{L^2}\\
\nn&\le&C\mu^{-\fr16}e^{-\fr{\mu^\fr13t}{44}}\|((\Dl b_{\mathrm{in}})_{\neq}, (\Dl u_\mathrm{in})_{\neq})\|_{L^2},
\eeq
and the estimate for $\pr_zb_{\neq}$ is similar, thus \eqref{bxz} holds. Next, we postpone the estimate of $u^1_\neq$ and prove \eqref{1.14} and \eqref{1.15} first, since $u^2_\neq$ and $u^3_\neq$ are easy to treat.  Recalling that ${\bf w}^2=\Dl u^2_\neq-\pr_y\dv u+\pr_xb$, then using \eqref{m-2}, \eqref{enhan13} and \eqref{enhan2}, we find that
\beq
\nn\|\Dl_{x,z}u^2_{\neq}\|_{L^2}&=&\|\Dl_{x,z}\Dl^{-1}({\bf w}^2_\neq-\pr_xb_\neq+\pr_y\dv u)\|_{L^2}\\
\nn&=&\|\Dl_{X,Z}\Dl^{-1}_{L}(W^2_\neq-B^1_\neq+D^2_\neq)\|_{L^2}\\
\nn&\le&\|m^{-1}(W^2_\neq-B^1_\neq+D^2_\neq)\|_{L^2}\\
\nn&\le&C\left(\|m^{-\fr34}W^2_\neq\|_{L^2}+\left\|m^{-\fr34}\sqrt{-\Dl_L}\fr{B^1_\neq}{\eps}\right\|_{L^2}+\|m^{-\fr34}D^2_\neq\|_{L^2}\right)\\
\nn&\le&C\mu^{-\fr13}e^{-\fr{\mu^\fr13t}{88}}\|((\Dl b_{\mathrm{in}})_{\neq}, (\Dl u_\mathrm{in})_{\neq})\|_{L^2},
\eeq
thus \eqref{1.14} holds. We would like to point out  that \eqref{enhan2} does not need to be involved in  the estimates of $\Dl_{x,z}u^3_\neq$ and $\pr_{xx}u^1$. More precisely, it follows from \eqref{m-2} and \eqref{enhan13} that
\beq
\nn\|\Dl_{x,z}u^3_{\neq}\|_{L^2}&=&\|\Dl_{x,z}\Dl^{-1}({ w}^3_\neq+\pr_z\dv u)\|_{L^2}\\
\nn&=&\|\Dl_{X,Z}\Dl^{-1}_{L}(W^3_\neq+D^3_\neq)\|_{L^2}\\
\nn&\le&\|m^{-1}(W^3_\neq+D^3_\neq)\|_{L^2}\\
&\le&Ce^{-\fr{\mu^\fr13t}{44}}\|((\Dl b_{\mathrm{in}})_{\neq}, (\Dl u_\mathrm{in})_{\neq})\|_{L^2},
\eeq
this gives the proof of \eqref{1.15}. For the proof of \eqref{prxxu1}, we shall use the incompressibility of $w=\Dl u-\nb\dv u$ to avoid the appearance of $w^1$ and give a refined estimate of $W^2_\neq$. In fact, by virtue of $\dv w=0$, \eqref{m-2} and \eqref{enhan13}, there holds
\beq\label{u1xx}
\nn\|\pr_{xx}u^1\|_{L^2}&=&\|\pr_{xx}\Dl^{-1}({ w}^1+\pr_x\dv u)\|_{L^2}\\
\nn&=&\|\pr_{xx}\Dl^{-1}{ \bf d}^1-\pr_{xy}\Dl^{-1}w^2-\pr_{xz}\Dl^{-1}w^3\|_{L^2}\\
\nn&=&\|\pr_{xx}\Dl^{-1}{ \bf d}^1-\pr_{xy}\Dl^{-1}{\bf w}^2+\pr_{xy}\Dl^{-1}{\bf b}^1-\pr_{xz}\Dl^{-1}w^3\|_{L^2}\\
\nn&=&\|\pr_{XX}\Dl^{-1}_{L}D_{\neq}^1-\pr_{XY}^{L}\Dl^{-1}_{L}W_{\neq}^2+\pr^{L}_{XY}\Dl^{-1}_{L}B_{\neq}^1-\pr_{XZ}\Dl^{-1}_{L}W_{\neq}^3\|_{L^2}\\
\nn&\le&\|m^{-1}(D^1_{\neq}+W^3_{\neq}+\sqrt{-\Dl_L}B^1_{\neq})\|_{L^2}+\|m^{-\fr12}W^2_{\neq}\|_{L^2}\\
&\le&Ce^{-\fr{\mu^\fr13t}{44}}\|((\Dl b_{\mathrm{in}})_{\neq}, (\Dl u_\mathrm{in})_{\neq})\|_{L^2}+\|m^{-\fr14}W^2_{\neq}\|_{L^2}.
\eeq
We proceed  via the equation $\eqref{Lx'}_3$ to bound $\|m^{-\fr14}W^2_{\neq}\|_{L^2}$. To this end, let us define 
\[
\tl{M}:=e^{\fr{\mu^\fr13t}{88}}m^{-\fr14}m_2^{-1}{\bf 1}_{k\neq0}.
\]
Similar to \eqref{e-W2}, we have
\beq\label{e-tlW2}
\nn&&\fr{1}{2}\fr{d}{dt}\|\tl{M}W^2\|^2_{L^2}+\fr14\left\|\sqrt{\fr{\pr_tm}{m}}\tl{M}W^2\right\|^2_{L^2}\\
\nn&&+\left\|\sqrt{\fr{\pr_tm_2}{m_2}}\tl{M}W^2\right\|^2_{L^2}+\mu\left\|\tl{M}\sqrt{-\Dl_L}W^2\right\|^2_{L^2}\\
\nn&=&\fr{\mu^\fr13}{44}\|\tl{M}W^2\|^2_{L^2}+\mu\left\la \tl{M}\sqrt{-\Dl_L}B^1, \tl{M}\sqrt{-\Dl_L}W^2\right\ra\\
&\le&\fr{\mu^\fr13}{44}\|\tl{M}W^2\|^2_{L^2}+\fr{\mu}{4}\|\tl{M}\sqrt{-\Dl_L}W^2\|_{L^2}^2+ \mu\|\tl{M}\sqrt{-\Dl_L}B^1\|_{L^2}^2.
\eeq
Thanks to \eqref{claim}, the first two terms on the right hand side of \eqref{e-tlW2} can be absorbed by the left hand side. On the other hand, in view of \eqref{m-1} and \eqref{enhan13}, we find that
\[
\|\tl{M}\sqrt{-\Dl_L}B^1\|_{L^2}^2\le C\mu^{-\fr23}\left\|e^{\fr{\mu^\fr13t}{88}}m^{-\fr34}\sqrt{-\Dl_L}B^1_{\neq}\right\|_{L^2}^2\le C\mu^{-\fr23}e^{-\fr{\mu^\fr13t}{44}}\|((\Dl b_{\mathrm{in}})_{\neq}, (\Dl u_\mathrm{in})_{\neq})\|_{L^2}^2.
\]
Substituting this into \eqref{e-tlW2}, and integrating with respect to the time variable, one deduces that
\beq
\nn&&\|\tl{M}W^2(t)\|^2_{L^2}+\mu^\fr13\|\tl{M}W^2\|^2_{L^2L^2}+\mu\left\|\tl{M}\sqrt{-\Dl_L}W^2\right\|^2_{L^2L^2}\\
\nn&\le&C\|W^2_{\neq}(0)\|_{L^2}^2+C\|((\Dl b_{\mathrm{in}})_{\neq}, (\Dl u_\mathrm{in})_{\neq})\|_{L^2}^2\\
&\le&C\|((\Dl b_{\mathrm{in}})_{\neq}, (\Dl u_\mathrm{in})_{\neq})\|_{L^2}^2,
\eeq
which implies that
\be
\|m^{-\fr14}W^2_\neq(t)\|_{L^2}\le Ce^{-\fr{\mu^\fr13t}{88}}\|((\Dl b_{\mathrm{in}})_{\neq}, (\Dl u_\mathrm{in})_{\neq})\|_{L^2}.
\ee
It follows from this and \eqref{u1xx} that \eqref{prxxu1} holds. We complete the proof of Theorem \ref{coro1}.

Now we  turn to the proof of  \ref{thm-endis}. To this end, recalling  the definitions of $(B^i, D^i, W^i), i=2, 3$, we have
\beqno
&&\left\|\frac{\nb \nb_{x,z}b_\neq}{\eps}\right\|^2_{L^2}+\|\nb_{x,z}\dv u_\neq\|^2_{L^2}+\|{\bf w}^2_\neq\|^2_{L^2}+\|w^3_\neq\|^2_{L^2}\\
&=&\sum_{i=1,3}\left(\left\|\sqrt{-\Dl_L}\frac{B^i_\neq}{\eps}\right\|^2_{L^2}+\|D^i_\neq\|^2_{L^2}\right)+\|W^2_\neq\|^2_{L^2}+\|W^3_\neq\|^2_{L^2}.
\eeqno
Then   from \eqref{m-1},  \eqref{enhan13}, the fact ${\bf w}^2=\Dl u^2-\pr_y\dv u+\pr_xb$, and
\beno
\|\Dl u^2_\neq-\pr_y\dv u_\neq\|_{L^2}\le\|\Dl u^2_\neq-\pr_y\dv u_\neq+\pr_xb\|_{L^2}+\eps\left\|\fr{\nb\pr_xb}{\eps}\right\|_{L^2},
\eeno
we find that \eqref{enhan-13} holds.
Similarly,  \eqref{enhan-2} is a consequcence of \eqref{m-1},  \eqref{enhan2} and Remark \ref{rem4.1}. This completes the proof of Theorem \ref{thm-endis}.

Finally, we are left to prove Theorem \ref{thm-0}. In fact, \eqref{en-0} is nothing but \eqref{en0} in Proposition \ref{prop-en0}. Combining  \eqref{decay1} with \eqref{de1} and \eqref{de2}, and using \eqref{visco}, one easily deduces that \eqref{decay0} holds. From \eqref{liftup} and \eqref{visco}, we get \eqref{lift-up}. The proof of Theorem \ref{thm-0} is completed.

\bigbreak
\noindent{\bf Acknowledgments}
L. Zeng  is supported by the  Postdoctoral Science Foundation of China under Grant  8206400030. Z. Zhang is partially supported by  NSF of China  under Grant 11425103. R. Zi is partially  supported by NSF of China under  Grants 11871236 and 11971193.
\bigbreak
\noindent{\bf Conflict of interest} The authors declare that they have no conflict of interest.

\bigbreak

\end{document}